\documentclass[10pt,reqno]{amsart}
\usepackage{amsmath, latexsym, amsfonts, amssymb,amsthm, amscd,epsfig,enumerate,bm}
\usepackage{subfigure,color, float}

\advance\hoffset -.75cm

\usepackage{xcolor}

\definecolor{refkey}{gray}{.75}
\definecolor{labelkey}{gray}{.75}

\newcommand{\R}{\mathbb R}

\newcommand{\N}{\mathbb N}

\newcommand{\E}{\mathbb E}

\newcommand{\diff}{\mathrm{d}}

\newcommand{\cH}{\mathcal{H}}

\newcommand{\pr}{\mathbb P}

\newcommand{\ident}{{\mathchoice {\rm 1\mskip-4mu l} {\rm 1\mskip-4mu l}
{\rm 1\mskip-4.5mu l} {\rm 1\mskip-5mu l}}}

\newcommand{\gegerm}{{\,\ge_{\textrm{germ}\,}}}

\newcommand{\gepgf}{{\,\ge_{\textrm{pgf}\,}}}

\newcommand{\boldmu}{{\bm{\mu}}}
\newcommand{\boldnu}{{\bm{\nu}}}

\newtheorem{teo}{Theorem}[section]
\newtheorem{lem}[teo]{Lemma}
\newtheorem{cor}[teo]{Corollary}
\newtheorem{rem}[teo]{Remark}
\newtheorem{pro}[teo]{Proposition}
\newtheorem{defn}[teo]{Definition}
\newtheorem{exmp}[teo]{Example}

\newtheorem{assump}[teo]{Assumption}

\title
{Strong survival and extinction for multitype branching processes via a new order for generating functions}

\author[D.~Bertacchi]{Daniela Bertacchi}
\address{D.~Bertacchi, Dipartimento di Matematica e Applicazioni,
Universit\`a di Milano--Bicocca,
via Cozzi 53, 20125 Milano, Italy.}
\email{daniela.bertacchi\@@unimib.it}

\author[F.~Zucca]{Fabio Zucca}
\address{F.~Zucca, Dipartimento di Matematica,
Politecnico di Milano,
Piazza Leonardo da Vinci 32, 20133 Milano, Italy.}
\email{fabio.zucca\@@polimi.it}

\date{}

\begin{document}

\begin{abstract}
We consider general discrete-time multitype branching processes on a countable set $X$. 
According to these processes, a particle at $x\in X$ generates a random number of children and places them at (some of) the sites of $X$, 
not necessarily independently nor with the same law at different starting vertices $x$.
We introduce a new type of stochastic ordering of multitype branching processes, generalizing the germ order 
 introduced by Hutchcroft in \cite{cf:Hut2022},
which relies on the generating function of the process. We prove that given two multitype branching processes with law $\boldmu$ and $\boldnu$ respectively, with $\boldmu\ge\boldnu$, then in every set where there is survival according to $\boldnu$, there is survival also according to $\boldmu$. Moreover, in every set where there is strong survival according to $\boldnu$, there is strong survival also according to $\boldmu$, provided that the supremum of the global extinction probabilities, for the $\boldnu$-process, taken over all starting points $x$, is strictly smaller than 1. 
New conditions for survival and strong survival for inhomogeneous multitype branching processes are provided.
We also extend a result of Moyal \cite{cf:Moyal}
which claims that, under some conditions, the global extinction probability for a multitype branching process is the
only fixed point of its generating function, whose supremum over all starting coordinates may be smaller than 1.
\end{abstract}


\maketitle
\noindent {\bf Keywords}: branching random walk, multitype branching process, generating function, fixed point, extinction probability vector, germ order, pgf order, strong survival, maximal and minimal displacement.

\noindent {\bf AMS subject classification}: 60J80, 60J10.

\section{Introduction}
\label{sec:intro}

The multitype branching process (or briefly $\textit{MBP}$) on an at most countable set $X$
is a process which describes the evolution of a population breeding and dying on $X$, where the elements of $X$ can be seen as \textit{types} or  \textit{positions} of the individuals of the population. Throughout our paper we stick with the second interpretation and we consider $X$ as the space  where the dynamics take place.
Another common name for this process is \textit{branching random walk}, although some authors reserve this denomination for the case when $X$ is endowed with a graph structure.

A general MBP is defined once we fix the reproduction law $\boldmu=\{\mu_x\}_{x\in X}$ (see Section \ref{subsec:process} for details).
All particles at site $x$ breed and place children according to $\mu_x$, which incorporates not only information about how many
the children are, but also about where they are sent to live. In this sense particles do not walk, but there is a random walk of
the population as a whole.

The \textit{branching process} can be seen as  a particular case of the MBP, where $X$ is reduced to a singleton and the only
information needed is the reproduction law $\rho$ defined on $\N$.
A natural way to define a MBP on $X$ is to couple a family of branching processes, given by reproduction laws $\{\rho_x\}_{x\in X}$,
and a random walk with transition matrix $P$ on $X$. Each individual at $x$ has a $\rho_x$-distributed number of offspring, which are 
independently dispersed according to the random walk. We call this kind of process MBP \textit{with independent diffusion}.
We remark that for general MBP, the dispersal of the progeny may not be independent nor based on a random walk (for instance we may
place two children at a given vertex with probability $p$ and one child at each of a couple of other vertices with probability $1-p$).

We are interested in the long-term behaviour of the process in fixed subsets of $X$.
In the long run, for any $A\subseteq X$,
a MBP starting with one individual at $x\in X$ can go extinct in $A$ (no individuals alive in 
$A$ from a certain time on) or survive in $A$ (infinitely many visits to $A$).
If the probability of extinction in $A$ is equal to 1, we say that there is \textit{extinction} in $A$, and we say that
there is \textit{survival} in $A$ otherwise.
There is \textit{global survival} when there is survival in $X$  and we have \textit{strong survival} in $A$ when, 
conditioned on global survival, there is survival in $A$ almost surely.

Clearly, the probability of extinction in $A$ depends on the starting vertex $x$. Then, letting $x$ vary 
in $X$, we get an extinction probability vector the we denote by $\mathbf{q}(A)$. If we allow $A$ to vary among the subsets of $X$, we have
the family of all extinction probability vectors of the MBP.

For the branching process, it is well known that the long-term behaviour and the extinction probability are linked with
the generating function of $\rho$, $G(z):=\sum_n\rho(n)z^n$, $z\in[0,1]$. 
Provided that the process is nontrivial, that is $\rho(1)<1$, this generating function has at most two fixed points: the extinction probability and 1.
In the case of a general MBP it is possible to define a (multi-dimensional) generating function which plays a similar role, but as 
soon as $X$ is not finite, the situation gets far more complex: there might be infinitely many fixed points and infinitely many extinction probability
vectors (see Section \ref{subsec:genfun}); moreover there can be fixed points that are not extinction probability vectors.
It is still true, however, that the vector $\mathbf 1$ is always a fixed point of the generating function and the \textit{global extinction}
probability vector (that is, the probability of extinction in the whole space $X$) is always the smallest fixed point.

The fact that the generating function of the process contains all the information on its behaviour is exploited
in the main result of the present paper.
In \cite{cf:Hut2022} the author focussed on MBPs with independent diffusion and 
reproduction law $\rho$ equal for all sites and 
introduced a new stochastic ordering. This order is named \textit{germ order} and is based on a comparison between the 
one-dimensional generating
functions of the reproduction laws. 
The author was able to compare MBPs which are defined by the same underlying random walk $P$ on $X$ and differ only by the
reproduction law, which is constant along $X$.

We define the germ order for general MBPs which extends the one in \cite{cf:Hut2022}: the proof of this fact can be found in Proposition~\ref{pro:germindepdiff} (see also the discussion preceding the proposition itself).
Then we extend \cite[Theorem 1.3]{cf:Hut2022}, by proving the following result
(a more precise statement is given by Theorem \ref{th:germorder}).
\begin{teo}\label{th:germorderintro}
	Let   $\boldmu$ and $\boldnu$ be the law of two MBPs on a countable space $X$ and let $\boldmu \gegerm \boldnu$.
	\begin{enumerate}
		\item 
		In any set where there is     $\boldnu$-survival, there is $\boldmu$-survival.
		\item If the supremum of the global $\boldnu$-extinction probabilities, over all starting coordinates, is smaller than 1,
		then in any set where there is $\boldnu$-strong survival, there is
		$\boldmu$-strong survival.
	\end{enumerate}
\end{teo}
The assumption in the second part of the statement may appear technical at first glance, but as discussed in Example~\ref{exmp:continuoustime} it cannot be removed. Moreover, it is worth remarking that
under very mild conditions,  among all fixed points, only the global extinction probability vector may satisfy this condition.
Indeed we extend a result of \cite{cf:Moyal}, which states that, under certain conditions, the global extinction probability vector is the only
fixed point which may have coordinates bounded from above by some $\delta<1$. 
We are able to prove, in Theorem~\ref{th:moyal2}, that under no conditions at all,  the global extinction probability vector is the only extinction probability vector 
which can have this property. Moreover, if a mild condition is satisfied, it is also the only fixed point
with supremum different from 1.
This result allows us to extend the original proof
of \cite[Theorem 1.3]{cf:Hut2022} to the case of general MBPs.

The paper is organized as follows. 
Section~\ref{sec:basic} is devoted to the basic definitions and is divided in subsections.
In Section~\ref{subsec:germorder} we introduce the generating function of a family of measures and the general \textit{germ order} along with its main properties.
We recall the usual stochastic order and the pgf order for measures.
The germ order is weaker than the pgf order, which in turn, is weaker than the usual stochastic order.
This definition of germ order extends the one given in
\cite{cf:Hut2022}.
In Section~\ref{subsec:process}
we formally introduce the MBP on a countable space $X$.
In Section~\ref{subsec:survival} we define survival, strong survival and extinction in the whole space $X$ and in subsets $A\subseteq X$.
Section~\ref{subsec:genfun} is devoted to the properties of the generating function $G$ of the MBP, which is already known to be
useful since extinction probabilities are (some of) its fixed points.
In particular, Proposition~\ref{pro:closed} shows that fixed points of $G$ can be found by iterating the function itself on suitable starting vectors.  
As shown in Section~\ref{subsec:martingale}, fixed points can be used to construct special martingales, which have interesting properties.
These properties are exploited in the proof of Theorem~\ref{th:moyal2}, which is
 the main result of Section~\ref{sec:Moyal}. This theorem
 shows that, given a generic (not necessarily irreducible) MBP, if there exists $\delta<1$ such that the
 probability of extinction in $A$, starting from any $x\in X$, is smaller than $\delta$, then $\mathbf{q}(A)=\mathbf{q}(X)$. The same 
can be said for any fixed point $\mathbf z$: if all its coordinates are small than $\delta$ for some $\delta<1$, then
$\mathbf{z}=\mathbf{q}(X)$, provided that
 the MBP satisfies a mild sufficient condition.
We show that without this condition, there are examples where the property does not hold for fixed points 
(see Examples~\ref{exmp:moyal1} and \ref{exmp:moyal2}).
Section~\ref{sec:germ} is devoted to 
the relation between survival (resp.~strong survival) for two MBPs satisfying the germ order.
The main result of the section, Theorem~\ref{th:germorder}, deals with the germ order, while Theorem~\ref{th:pgforder} deals with the pgf order. 
For two MBPs with independent diffusion we find a condition equivalent to germ-order (see Proposition~\ref{pro:germindepdiff}). The results of this section generalize the results in \cite{cf:Hut2022}.  
As explained in details at the end of the section, by using Theorem~\ref{th:germorder} and Proposition~\ref{pro:germindepdiff},  we are able to prove new and powerful conditions for survival and strong survival for inhomogeneous MBPs.
All the proofs, along with technical lemmas can be found in Section~\ref{sec:proofs}.  The final Appendix 
is devoted to the construction of $\mathbb{R}^X$ as a Polish space which is essential for coupling processes stochastically ordered in the classical way.

\section{Basic definitions and properties}
\label{sec:basic}

\subsection{Generating function orders}\label{subsec:germorder}

Given an at most countable set $X$ and a set $Y$
we consider a family of measures $\boldmu=\{\mu_y\}_{y \in Y}$
defined on the (countable) 
measurable space 
$S_X:=\{f:X \to \N\colon |f|<+\infty\}$ equipped with the $\sigma$-algebra $2^{S_X}$
where 
$|f|:=\sum_yf(y)<+\infty$; throughout this paper we denote by $\N$ the set of natural numbers including $0$. 
An interpretation is the following: suppose that an individual marked with a label $y$ has a random number of items to place in a space $X$, then $\mu_y(f)$ represents the probability that there are $f(x)$ items placed at $x$ (for all $x \in X$).

To the family $\{\mu_y\}_{y \in Y}$, we associate the following generating function $G_\boldmu:[0,1]^X \to [0,1]^Y$,
\begin{equation}\label{eq:genfun}
	G_\boldmu({\mathbf{z}}|y):= \sum_{f \in S_X} \mu_y(f) \prod_{x \in X} {\mathbf{z}}(x)^{f(x)}, 
\end{equation}
where $G_\boldmu({\mathbf{z}}|y)$ is the $y$ coordinate of $G_\boldmu({\mathbf{z}})$.
The family $\{\mu_y\}_{y \in Y}$ is uniquely determined by $G_\boldmu$ (see for instance \cite[Section 2.3]{cf:BZ14-SLS} or \cite[Section 2.2]{cf:BZ2017} and Lemma~\ref{lem:uniqueness}
). 
Henceforth, when it is not misleading, we write $G$ instead of $G_\boldmu$.  We define $\phi^\boldmu_y(t):=G_\boldmu(t\mathbf{1}|y)$ for $t \in [0,1]$ and $y \in Y$ (sometimes we write $\phi_y$ instead of $\phi^\boldmu_y$) where $\mathbf{1}(x):=1$ for all $x \in X$ (similarly we define $\mathbf{0} \in [0,1]^X$ as $\mathbf{0}(x):=0$ for all $x \in X$). Note that, if
\begin{equation}\label{eq:total}
	\rho_y(n):= \mu_y(f\colon |f|=n),
\end{equation}
then $\phi_y$ is the one-dimensional generating function of $\rho_y$.
The topological properties of $G_\boldmu$ are described in the following proposition; in particular, we define $\|\mathbf{z}\|_\infty:=\sup_{x \in C} |\mathbf{z}(x)|$ the restriction of the norm of $l^\infty(C)$ to $[0,1]^C$ (for $C\in\{X,Y\}$). The (partially ordered) spaces $[0,1]^X$ and $[0,1]^Y$ can be equipped with two useful topologies: the product (or \textit{pointwise convergence}) topology  and the finer topology arising from the metric $d(\mathbf{z}, \mathbf{v}):=\|\mathbf{z}- \mathbf{v}\|_\infty$ .

\begin{pro}\label{pro:Gtopology}
	Let us consider the generating function $G_\boldmu$ defined by eq.~\eqref{eq:genfun}.
	\begin{enumerate}
		\item $G$ is 
		non-decreasing
		with respect to the usual partial order of $[0,1]^X$ and $[0,1]^Y$.
		\item $G$ is continuous with respect to the \textit{pointwise convergence topology} of $[0,1]^X$ and $[0,1]^Y$.
		\item If the family $\{\rho_y\}_{y \in Y}$ is tight then $G$ is uniformly continuous with respect to the $\|\cdot\|_\infty$-topologies of  $[0,1]^X$ and $[0,1]^Y$.
	\end{enumerate}
\end{pro}

Given a family $(\rho_y)_{y \in Y}$ of measures on $\mathbb{N}$ and a nonnegative stochastic matrix $P=(p(y,x))_{y \in Y, x \in X}$ (where $\sum_{x \in X} p(y,x)=1$ for all $y \in Y$) then
we say that $\boldmu$ is an \emph{multinomial} family of measures if
\begin{equation}\label{eq:particular1}
\mu_y(f)=\rho_y \left (\sum_{x \in X} f(x) \right )\frac{(\sum_{x \in X} f(x))!}{\prod_{x \in X} f(x)!} \prod_{x \in X} p(y,x)^{f(x)},
\quad \forall f \in S_X.
\end{equation}
If we use the above interpretation of the family $\boldmu$ then, in the case of a multinomial family, an individual marked with the label $y$ draws a random number $n$ of items (according to $\rho_y$) and places each one independently in $X$ according to the distribution $p(y, \cdot)$. In the language of MBPs this is called \textit{independent diffusion}, see Section~\ref{subsec:process}.

It is easy to prove that for a multinomial family $\boldmu$
\begin{equation}\label{eq:Gindepdiff}
G(\mathbf{z}|y)=\phi^\boldmu_y (P\mathbf{z}(y)), \quad \forall y \in Y, \, \mathbf{z} \in [0,1]^X,
\end{equation}
where $P\mathbf{z}(y)=\sum_{x \in X} p(y,x)\mathbf{z}(x)$.
In this case,
$\phi^\boldmu_y(t)=G(t\mathbf{1}|y)=\sum_{i \in \mathbb{N}} \rho_y(i) t^i$  is the generating function of $\rho_y$. Indeed, by using the definition of $G$ and equation~\eqref{eq:particular1}
\[
\begin{split}
G(\mathbf{z}|y)&=\sum_{f \in S_X} 
\rho_y \left (\sum_{x \in X} f(x) \right )\frac{(\sum_{x \in X} f(x))!}{\prod_{x \in X} f(x)!} \prod_{x \in X} p(y,x)^{f(x)} 
\prod_{x \in X} \mathbf{z}(x)^{f(x)}\\	
&=\sum_{n \in \mathbb{N}} \rho_y (n) \sum_{\stackrel{f \in S_X}{\sum_{x \in X}f(x)=n} } 
\frac{n!}{\prod_{x \in X} f(x)!} \prod_{x \in X} \big (p(y,x) \mathbf{z}(x) \big )^{f(x)} \\
&=\sum_{n \in \mathbb{N}} \rho_y (n) \Big (\sum_{x \in X} 
p(y,x) \mathbf{z}(x) \Big )^n. 
\end{split}
\]

\begin{defn}\label{def:ordering}
Let $\boldmu:=\{\mu_y\}_{y \in Y}$ and $\boldnu:=\{\nu_y\}_{y \in Y}$ be two families of measures on 
$S_X$.
Let $G_{\boldmu}$ and $G_\boldnu$ be the associated generating functions.
\begin{enumerate}
		 \item $\boldmu \succeq \boldnu$ if and only if 
	$\mu_y \succeq \nu_y$ for all $y \in Y$, that is, if and only if
	given a non-decreasing measurable function $F\colon S_X\to \R$, we have
	$\int F \diff \mu_y \ge \int F \diff \nu_y$ for all $y \in Y$ such that the integrals are well defined.
	\item $\boldmu \gepgf \boldnu$ if and only if
	$G_\boldmu(\mathbf{z}) \le G_\boldnu(\mathbf{z})$ for all $\mathbf{z} \in [0,1]^X$.
	\item $\boldmu \gegerm \boldnu$ if and only if there exists $\delta \in [0,1)$
	$G_\boldmu(\mathbf{z}) \le G_\boldnu(\mathbf{z})$ for all $\mathbf{z} \in [\delta,1]^X$.
\end{enumerate}
If $\# Y=1$, that is, $\boldmu=\{\mu\}$ and $\boldnu=\{\nu\}$,
then we simply write $\mu \gepgf \nu$ and $\mu \gegerm \nu$.
\end{defn}

We observe that $\boldmu \succeq \boldnu \ \Rightarrow \boldmu \gepgf \boldnu \ \Rightarrow \ \boldmu \gegerm \boldnu$, but the reverse implications do not hold. 
Clearly $G_\boldmu(\mathbf{z}) \le G_\boldnu(\mathbf{z})$ if and only if $G_\boldmu(\mathbf{z}|y) \le G_\boldnu(\mathbf{z}|y)$ for all $y \in Y$; thus,
$\boldmu \gegerm \boldnu$ (with a certain $\delta <1$ ) if and only if $\mu_y \gegerm \nu_y$ for all $y \in Y$ (with the same $\delta <1$).

We recall that for real-valued measures (that is, when $Y$ is a singleton),  $\mu \succeq \nu$ is equivalent to the existence of
two random variables $\eta$, $\zeta$ with laws $\mu$ and $\nu$ respectively, such that $\eta\ge\zeta$ a.s. (this construction
is usually referred as an ordered coupling).
This result can be extended to measures on partially ordered, compact metric spaces (\cite[Theorem 2.4]{cf:Liggett}) and 
to measures on partially ordered Polish spaces (see for instance \cite[Theorem 1]{cf:KKO}). 
It is not difficult to show that
$\mathbb{R}^X$, with a suitable finite metric, is a partially ordered Polish space. 
\smallskip

The following result shows that $\gegerm$ is a partial order. 
\begin{pro}\label{pro:germordertrueorder} The binary relation
$\gegerm$ is a partial order.
\end{pro}
 We note that if $\# X=\# Y=1$, and $G_\boldmu$ and $G_\boldnu$ admit an holomorphic extension in a neighborhood of $1$, then
$\gegerm$ defines a total order. Indeed in this case if there is no $\delta\in[0,1)$ such that 
$G_\boldmu(z)<G_\boldnu(z)$ for all $z\in (\delta,1)$ or  $G_\boldmu(z)>G_\boldnu(z)$, then the two
 functions coincide 
 by \cite[Theorem 10.18]{cf:Rudin}.
The existence of an holomorphic extension of the generating functions is no longer sufficient, as soon as $X$ or $Y$ has at least cardinality
2.
Indeed if
 $\#Y \ge 2$ and the total offspring distributions $\{\rho_y\}_{y \in Y}$ are not constant with respect to $y$, 
 then $\gegerm$ is clearly not a total order. If $\#Y=1$ and $\#X=2$ 
  see the example after Proposition~\ref{pro:germindepdiff}.

Note that the MBPs discussed in  \cite{cf:Hut2022} are described by a multinomial family of measures, where $\rho_y \equiv \rho$ does not depend on $y$. The definition of germ order in  \cite{cf:Hut2022} concerns the generating function of $\rho$. One may extend this definition to general multinomial families by comparing, site by site, the corresponding generating function of $\rho_y$. This is not our definition (which is based on a multidimensional generating function), however in the case of multinomial measures the two definitions are equivalent as the following proposition shows. 


\begin{pro}\label{pro:germindepdiff}
Suppose that $\boldmu$ and $\boldnu$ are two families of measures 
and let us define $\phi_y(t):=G_\boldmu(t \mathbf{1}|y)$ for all $y \in Y$ and $t \in [0,1]$.
Consider the following for any fixed $\delta<1$:
\begin{enumerate}
	\item $G_\boldmu(\mathbf{z}) \le G_\boldnu(\mathbf{z})$ for all $\mathbf{z} \in [\delta,1]^X$;
	\item $\phi_y^\boldmu(t) \le \phi_y^\boldnu(t)$ for all $t \in [\delta,1]$ and all $y \in Y$.
\end{enumerate}
Then $(1) \Rightarrow (2)$. Moreover if $\boldmu$ and $\boldnu$ are multinomial families with the same matrix $P$ (see equations~\eqref{eq:particular1} and \eqref{eq:Gindepdiff}), 
then $(2) \Rightarrow (1)$.	
\end{pro}

We note that, if the families are not multinomial, then in the previous proposition, (2) does not imply (1), even when $X$ and $Y$ are finite. Take for instance $X=\{1,2\}$, $Y=\{a\}$ and
\[
G_\boldmu(z_1,z_2|a):=
\frac56 z_1z_2 +\frac16
\qquad
G_\boldnu(z_1,z_2|a):=
\frac45 \Big (\frac{5z_1+z_2}6 \Big )^2 +\frac15\\
\]
Clearly $(2)$ holds for $\delta=0$ indeed, for all $t \in [0,1)$ we have
\[
G_\boldmu(t,t|a)=
\frac56 t^2 +\frac16
<
\frac45 t^2 +\frac15
=
G_\boldnu(t,t|a)
\]
nevertheless
\[
G_\boldmu(t,1|a)=\frac{5t}6 +\frac16
> \frac{(5t+1)^2}{45}+\frac15=
G_\boldnu(t,1|a)
\]
for all $t \in (1/10\, ,\, 1)$; thus $G_\boldmu$ and $G_\boldnu$ are incomparable. 

\subsection{The MBP}\label{subsec:process}

Henceforth, if not otherwise stated, we assume that $Y=X$, where $X$ is a countable space (the label coincides with the position).
In this case, the family $\boldmu=\{\mu_x\}_{x \in X}$ of probability measures
on the (countable) measurable space $(S_X,2^{S_X})$ induces a discrete-time MBP on $X$.
This is a process $\{\eta_n\}_{n \in \N}$, 
where $\eta_n(x)$ is the number of particles alive
at $x \in X$ at time $n$. 
The dynamics is described as follows:
a particle of generation $n$, at site $x\in X$, lives one unit of time;
after that, a function $f \in S_X$ is chosen at random according to the law $\mu_x$.
This function describes the number of children and their positions, that is,
the original particle is replaced by $f(y)$ particles at
$y$, for all $y \in X$. The choice of $f$ is independent for all breeding particles.

An explicit construction is the following: given a family $\{ f_{i,n,x}\}_{i,n \in \N, x \in X}$
of independent $S_X$-valued random variable such that, for every $x \in X$, $\{f_{i,n,x}\}_{i,n \in \N}$ have the common law $\mu_x$, then
the discrete-time MBP $\{\eta_n\}_{n \in \N}$ is defined
iteratively as follows
\begin{equation}\label{eq:evolBRW}
\eta_{n+1}(x)=\sum_{y \in X} \sum_{i=1}^{\eta_n(y)} f_{i,n,y}(x)=
\sum_{y \in X} \sum_{j=0}^\infty \ident_{\{\eta_n(y)=j\}} \sum_{i=1}^{j} f_{i,n,y}(x)
\end{equation}
starting from an initial condition $\eta_0$. 
The actual canonical construction can be carried on, by using Kolmogorov's Theorem, in such a way that the probability space and the process $\{\eta_n\}_{n \in \mathbb{N}}$ are fixed, while the probability measure depends on the starting configuration and the family $\boldmu$. When the initial configuration is $\eta$, then the corresponding probability measure is denoted by $\pr^\eta_\boldmu$  and the expectation by $\E^\eta_\boldmu$.
In the particular case when the initial state is
one particle at $x$, namely $\eta=\delta_x$ a.s.,  we write  $\pr^x_\boldmu$ and $\E^x_\boldmu$. 
When $\boldmu$ is fixed, we avoid the subscript $\boldmu$ in the above notations. Similarly, when a result holds for every initial condition $\eta$ (or when the initial condition is fixed) we avoid the superscripts $\eta$ and $x$. 
We denote the MBP by $(X,\boldmu)$; if needed, the
initial value will be indicated each time.
Clearly, $(X,\boldmu)$ is a Markov chain with absorbing state $\mathbf 0$,
the configuration with no particles at all sites.
We denote by $\{\mathcal{F}_n\}_{n \in \N}$ the filtration associated to the process, namely, $\mathcal{F}_n:=\sigma \Big (f_{i,j,x} \colon i,j \in \N, j < n, \ x \in X \Big )$. Note that $\mathcal F_0$ is the trivial $\sigma$-algebra. By using~\eqref{eq:evolBRW} it is easy to see that the MBP is \textit{adapted} to $\{\mathcal{F}_n\}_{n \in \mathbb{N}}$, that is, $\eta_n$ is $\mathcal{F}_n$-measurable for every $n \in \mathbb{N}$.

The total number of children associated to $f$ is represented by the
function $\cH:S_X \rightarrow \N$ defined by $\cH(f):=\sum_{y \in X} f(y)$;
the associated law $\rho_x(\cdot):=\mu_x(\cH^{-1}(\cdot))$ is the law of the random number of children
of a particle living at $x$.
We denote by
 $m_{xy}:=\sum_{f\in S_X} f(y)\mu_x(f)$ 
 the expected number of children that a particle living
 at $x$ sends to $y$. It is easy to show that $\sum_{y \in X} m_{xy}=\bar \rho_x$
 where $\bar \rho_x$ is the expected value of the law $\rho_x$.
%


In particular, if $\rho_x$ does not depend on $x \in X$, we say that the MBP can be
\textit{projected on a branching process} (see \cite{cf:BZ4} for the definition and details, see also Remark~\ref{rem:lowestfixedpoint2}).
This is a particular case of the following definition where $V$ is a singleton.
\begin{defn}
	\label{def:locallyisomorphic} 
	A MBP $(X, \mu)$ is projected onto a MBP $(V,\nu)$ if there exists a
	surjective map $g:X\to V$ such that
	$
	\nu_{g(x)}(\cdot)=
	\mu_x\left(\pi_g^{-1}(\cdot)\right)
	$, 
	where $\pi_g:S_X \rightarrow S_V$ is defined as $\pi_g(f)(v)=\sum_{z\in g^{-1}(v)}f(z)$ for all $f\in S_X$, $v \in V$.
\end{defn}
It is possible to show that $(X,\mu)$ is projected onto $(V, \nu)$ if and only if, for all ${\mathbf{z}} \in [0,1]^V$ and $x \in X$,
$	G_X({\mathbf{z}} \circ g|x)=G_V({\mathbf{z}}|g(x))$.
The meaning of this definition is that, given $\{\eta_n\}$ a realization of $(X,\mu)$, then $\xi_n(v):=\sum_{z \in g^{-1}(v)} \eta_n(z)$ is a realization of $(V,\nu)$.

This is particularly relevant when $V$ is a finite set and
it is called $\mathcal{F}$-MBP
 (see \cite[Section 2.3]{cf:BZ2017},
 Remark~\ref{rem:lowestfixedpoint2} or \cite{cf:BZ,cf:Z1}
for the details on the properties of this projection map).
Examples are the so called \textit{quasi-transitive MBPs} 
(see \cite[Section 2.4, p.~408]{cf:BZ14-SLS} for the formal definition), where the action of the group of automorphisms of 
the MBP (namely, bijective maps preserving the reproduction laws) has a finite number $j$ of orbits:
the finite set onto which we project has cardinality $j$.
When there is just one orbit, then the MBP is called \textit{transitive} (which is thus a particular case of MBP projected on a
branching process). 
We note that in general, an $\mathcal{F}$-MBP does not need to be transitive nor quasi-transitive.

%

It is important to note that, for a generic MBP, 
the locations of the offsprings are not (necessarily) chosen independently, but they are assigned by the 
function $f\in S_X$.
We denote by  $P$ the \textit{diffusion matrix} with entries $p(x,y)=m_{xy}/\bar \rho_x$.
When the children are dispersed independently, they are placed according to $P$
and the process is called 
\textit{MBP with independent diffusion}:
in this case $\boldmu$ is a multinomial family (see equation~\eqref{eq:particular1}).


To  a generic discrete-time MBP we associate a directed graph $(X,E_\mu)$ where $(x,y) \in E_\mu$  
if and only if $m_{xy}>0$.
We say that there is a path from $x$ to $y$ of length $n$, and we write $x \stackrel{n}{\to} y$, if it is
possible to find a finite sequence $\{x_i\}_{i=0}^n$ (where $n \in \N$)
such that $x_0=x$, $x_n=y$ and $(x_i,x_{i+1}) \in E_\mu$
for all $i=0, \ldots, n-1$. Clearly $x \stackrel{0}{\to}x$ for all $x \in X$; if there exists $n \in \N$
such that $x \stackrel{n}{\to} y$, then we write $x \to y$.
 Whenever $x \to y$ and $y \to x$ we write $x \rightleftharpoons y$.
If the graph $(X,E_\mu)$ is \textit{connected}, 
then we say that the MBP 
is \textit{irreducible}. 
 In order to avoid trivial situations where particles have exactly one offspring almost surely, we assume
 henceforth the following.
 \begin{assump}\label{assump:1}
 For all $x \in X$ there is a vertex $y \rightleftharpoons x$ such that
 $\mu_y(f\colon  \sum_{w\colon w \rightleftharpoons y} f(w)=1)<1$.
\end{assump}

\subsection{Survival and extinction}\label{subsec:survival}

\begin{defn}\label{def:extinction-event} $\ $
We call \textsl{survival in $A \subseteq X$} the event 
\[
\mathcal S(A):=\Big \{\limsup_{n \to +\infty} \sum_{y \in A} \eta_n(y)>0\Big \},
\]
and we denote by $\mathcal E(A)=\mathcal S(A)^\complement$ the event that we call \textsl{extinction in $A$}. We define the extinction probability vector $\mathbf{q}(A)$ as $\mathbf{q}(x,A):=\pr^x(\mathcal{E}(A))$ for $x \in X$.
\end{defn}

It is important to note that, in the canonical construction, the events $\{\mathcal{E}(A), \mathcal{S}(A)\}_{A \subseteq X}$ and the corresponding random variables
$\{\ident_\mathcal{E}(A), \ident_\mathcal{S}(A)\}_{A \subseteq A}$ are fixed and do not depend on $\boldmu$ and the initial configuration $\eta$. The dependence on $\boldmu$ and $\eta$ is in the probability measure $\pr_\boldmu^x$.

\begin{defn}\label{def:survival} $\ $
\begin{enumerate}
 \item 
The process \textsl{survives in $A \subseteq X$}, starting from $x \in X$, if
\[
{\mathbf{q}}(x,A)
<1;
\]
otherwise the process \textsl{goes extinct in }$A$ (or dies out in $A$). 
\item
The process \textsl{survives globally}, starting from $x$, if
it survives in $X$.
 \item
 There is \textsl{strong survival in $A \subseteq X$}, starting from $x \in X$,
 if
 $ 
 {\mathbf{q}}(x,A)=\mathbf{q}(x,X) <1.
 $ 
\end{enumerate}
\end{defn}
\noindent
In the rest of the paper we use the notation ${\mathbf{q}}(x,y)$ instead of ${\mathbf{q}}(x, \{y\})$ for all $x,y \in X$.
It is worth noting that, in the irreducible case, for every $A \subseteq X$, the inequality ${\mathbf{q}}(x,A)<{1}$ holds for some $x \in X$ if and only if it
holds for every $x \in X$ (although it may be $\mathbf{q}(x,A) \neq \mathbf{q}(y,A)$ for some $x \neq y$).
For details and results on survival and extinction see for instance  \cite{cf:BZ4, cf:Z1}.

Note that in \cite{cf:Hut2022} the definition of transient set corresponds to our definition of a set where there is extinction starting from every site $x\in X$; while the definition of recurrent set is equivalent to our definition of a set where there is strong survival
starting from every site $x\in X$. Strong local survival has been studied by many authors in the last 15 years, see for instance \cite{cf:Gantert09, cf:BZ14-SLS}. There are examples of MBPs where there is non-strong survival on some finite sets (see \cite[Example 4.2]{cf:BZ14-SLS} or \cite[Corollaries 4.3 and 4.4]{cf:BZ2017}).

\subsection{Generating function of a MBP}\label{subsec:genfun}

The generating function of a MBP on $X$ is $G_\boldmu:[0,1]^X \to [0,1]^X$ defined by 
equation~\eqref{eq:genfun} (with $Y=X$).
 It is easy to show that for all $\mathbf{z} \le \mathbf{v}$, $t \mapsto G(\mathbf{z}+t(\mathbf{v}-\mathbf{z}))$ is a convex function and, in some cases, it is a strictly convex function (see \cite[Lemma 5.1]{cf:BZ14-SLS}); nevertheless, in general, 
 the function $G$ 
 is convex (see \cite[Section 3.1]{cf:BZ2017}). The generating function of the total number of children satisfies $\phi_x(t):=\sum_{n \in \mathbb{N}} \rho_x(n) t^n=G(t\mathbf{1}|x)$ for all $x \in X$ and $t \in [0,1]$.

As in the case of a branching process, extinction probabilities are fixed points
of $G$. The smallest fixed point is ${\mathbf{q}(X)}$:
more generally, given a solution of $G(\mathbf{z}) \le \mathbf{z}$, then $\mathbf z \ge {\mathbf{q}(X)}$.
Consider now the closed sets $F_G:=\{\mathbf{z} \in [0,1]^X \colon G(\mathbf{z})=\mathbf{z}\}$, 
$U_G:=\{\mathbf{z} \in [0,1]^X \colon G(\mathbf{z}) \le \mathbf{z}\}$ and $L_G:=\{\mathbf{z} \in [0,1]^X \colon G(\mathbf{z})\ge \mathbf{z}\}$;
clearly $F_G = U_G \cap L_G$. Moreover, by the monotonicity property, $G(U_G) \subseteq U_G$ and $G(L_G) \subseteq L_G$.
The iteration of $G$ produces sequences converging to fixed points.

\begin{pro}\label{pro:closed}
	Fix $\mathbf{z}_0 \in [0,1]^X$ and define, iteratively,
$\mathbf{z}_{n+1}:=G(\mathbf{z}_n)$ for all $n \in \N$. Suppose that  
 $\mathbf{z}_n \to \mathbf{z}$ as $n \to +\infty$ for some $\mathbf{z} \in [0,1]^X$. Then 
 $\mathbf{z} \in F_G$.  
 Moreover, fix $\mathbf{w} \in [0,1]^X$. 
\begin{enumerate}
 \item If $\mathbf{w} \in U_G$ 
 then $\mathbf{w} \ge \mathbf{z}_0$ implies $\mathbf{w} \ge \mathbf{z}$ (the converse holds for 
 $\mathbf{z}_0 \in L_G$).
 \item If $\mathbf{w} \in L_G$ 
 then $\mathbf{w} \le \mathbf{z}_0$ implies $\mathbf{w} \le \mathbf{z}$ (the converse holds for 
 $\mathbf{z}_0 \in U_G$).
\end{enumerate} 
 \end{pro}
 The proof is straightforward (see for instance \cite{cf:BZ2}). 
 The sequence $\{\mathbf{z}_n\}_{n \in \N}$ defined in the previous proposition converges 
 if $\mathbf{z}_0 \in L_G$ (resp.~$\mathbf{z}_0 \in U_G$): in that case $\mathbf{z}_n \uparrow \mathbf{z}$ (resp.~$\mathbf{z}_n \downarrow \mathbf{z}$) 
 for some $\mathbf{z} \in F_G$.

	We note that $\mathbf{q}(X)$ is not only the smallest fixed point of $G$, but also of any of its iterates $G^{(n)}$, 
	where $G^{(1)}:=G$ and $G^{(n+1)}:=G \circ G^{(n)}$ for every $n \ge 1$. Indeed
	it is known (see for instance 
	\cite{cf:Z1}) that $\mathbf{q}(X)=\lim_{i \to +\infty} G^{(i)}(\mathbf{0}) = \lim_{i \to +\infty} G^{(i\cdot n)}(\mathbf{0})$ for every $n \ge 1$. By Proposition~\ref{pro:closed}, since $\mathbf{0} \in L_G$ is the smallest point of $[0,1]^X$, the above sequence converges to the smallest fixed point of $G^{(n)}$ for all $n \ge 1$.

 Let us briefly address the question of the cardinality of the set of fixed points $F_G$ and its subset $\mathrm{ext}(G):=\{\mathbf{z} \in [0,1]^X\colon \mathbf{z}=\mathbf{q}(A),\, A \subseteq X\}$, that is, the set of extinction probability vectors; the question is relevant in the case of irreducible processes, otherwise it is very easy to find examples where these sets are finite or infinite. 
 It is clear that the cardinality of both sets is at most $\mathfrak{c}:=2^{\aleph_0}$ (where $\aleph_0$ is the cardinality of $\mathbb{N}$). Let us denote the cardinality of a set by $\#$.
 An example can be found in \cite{cf:BZ2017} where $\# F_G=\mathfrak{c}$ while $\# \mathrm{ext}(G)=2$; whence there are fixed points which are not extinction probability vectors. In \cite[Example 4.2]{cf:BZ14-SLS} there is an irreducible MBP where 
 $\# \mathrm{ext}(G) \ge 3$, in \cite{cf:BraHautHessemberg2} there is an example where $\# \mathrm{ext}(G) \ge 4$ and in \cite{cf:BZ2020} there is an example where $\# \mathrm{ext}(G)= \mathfrak{c}$. The question on the cardinality of $\mathrm{ext}(G)$ was completely solved in \cite{cf:BBHZ} where it has been shown that, for every choice of $N \in \mathbb{N} \cup \{\aleph_0, \mathfrak{c}\}$ there exists an irreducible MBP where the cardinality of $\mathrm{ext}(G)$ is $N$.

We recall that if the MBP has independent diffusion (that is, $\boldmu$ satisfies equation~\eqref{eq:particular1}), then $G_\boldmu$ satisfies equation~\eqref{eq:Gindepdiff} (again, $Y=X$). More precisely
\begin{equation}
 G(\mathbf{z}|x)=\phi_x (P\mathbf{z}(x)), \quad \forall x \in X, \, \mathbf{z} \in [0,1]^X,
\end{equation}
where $\phi_x(t):=\sum_{i \in \mathbb{N}} \rho_x(i) t^i$ is the generating function of the total number of children of a particle at $x$. 

\subsection{Useful super/submartingales}\label{subsec:martingale}
Given $\mathbf z\in[0,1]^X$ and $\mathbf w\in[0,+\infty)^X$, we define $\mathbf z^{\mathbf w}\in[0,1]$ as
\[
\mathbf z^{\mathbf w}:=\prod_{x\in X} \mathbf z(x)^{\mathbf w(x)}.
\]
Note that this infinite product always converges, being the limit of a nonincreasing sequence (for any choice of ordering of the elements in $X$). 

The first result gives an explicit expression of the conditional expectation of the above product in terms of the generating function of the process. 

\begin{lem}\label{lem:martingale}
For every $\mathbf z\in[0,1]^X$,  $m\ge0$, $k\ge1$ and for every initial condition $\eta$, we have
\[
\mathbb E^\eta\left[\mathbf z^{\eta_{m+k}}|\mathcal F_m\right] = (G^{(k)}(\mathbf z))^{\eta_m}, \ \pr^\eta\textrm{-a.s.}
\]
\end{lem}

The previous lemma and Doob's Martingale Convergence Theorem imply the following.

\begin{pro}\label{pro:martingale}
For every give initial state $\eta$, if
$\mathbf{z} \in L_G$
(resp.~$\mathbf{z} \in U_G$),
then $\E^\eta[\mathbf{z}^{\eta_{n+1}}|\mathcal{F}_n] \ge \mathbf{z}^{\eta_n}$
(resp.~$\E^\eta[\mathbf{z}^{\eta_{n+1}}|\mathcal{F}_n] \le \mathbf{z}^{\eta_n}$) for all $n \ge 0$. In particular
if $\mathbf{z} \in L_G \cup U_G$ then there exists a $[0,1]$-valued, $\mathcal{F}_\infty$-measurable random variable $W_\mathbf{z}$ such that, 
\[
\mathbf{z}^{\eta_n} \to W_\mathbf{z}, \quad \pr^\eta
\text{-a.s.}~\text{and in }L^p(\pr^\eta)\  
\forall p \ge 1.
\]
Moreover if $\mathbf{z} \in L_G $ (resp.~$\mathbf{z} \in U_G$) then $\E^\eta[W_\mathbf{z} | \mathcal{F}_n]\ge \mathbf{z}^{\eta_n}$   (resp.~$\E^\eta[W_\mathbf{z} | \mathcal{F}_n]\le \mathbf{z}^{\eta_n}$)  $\pr^\eta$-a.s.

\end{pro}

Note that, 
for every $\mathbf{z} \in [0,1]^X$, we have that
$\mathbf{z}^{\eta_n} \to 1$ on $\mathcal{E}(X)$;  whence 
$W_\mathbf{z}=1$, $\pr^\eta$-a.s.~on $\mathcal{E}(X)$.
Moreover, 
if $\mathbf{z} \in L_G$, $\E^\eta[W_\mathbf{z}] \ge \mathbf z^\eta$ and $\E^x[W_\mathbf{z}] \ge \mathbf z(x)$;
similarly if $\mathbf{z} \in U_G$, $\E^\eta[W_\mathbf{z}] \le \mathbf z^\eta$ and $\E^x[W_\mathbf{z}] \le \mathbf z(x)$.
For general $\mathbf z$,
Corollary~\ref{cor:martingale2} 
gives monotonicity of the limit $W_\mathbf{z}$ with respect to $\mathbf{z}$.
Corollary~\ref{cor:martingale} gives the limit of the martingale $\mathbf{z}^{\eta_n}$, when $\mathbf z=\mathbf{q}(A)$.
The submartingale plays a crucial role in the proof of Theorem~\ref{th:moyal2}. The proofs of the following corollaries can be found in Section~\ref{sec:proofs}.

\begin{cor}\label{cor:martingale}
 If 
 $A \subseteq X$, 
 then
 $\mathbf{q}(A)^{\eta_n} \to \ident_{\mathcal{E}(A)}$  $\pr^\eta$-a.s.~and in $L^p(\pr^\eta)$
for all $p \ge 1$.
\end{cor}

\begin{cor}\label{cor:martingale2}
 If $\mathbf{z}$ and $\mathbf{v}$ are two fixed points, then the following are equivalent
 \begin{enumerate}
  \item $\mathbf{z} \ge  \mathbf{v}$.
  \item $\pr^\eta(W_\mathbf{z} \ge
  W_\mathbf{v})=1$ for every initial condition $\eta \in S_X$.
  \item $\pr^x(W_\mathbf{z} \ge
  W_\mathbf{v})=1$ for every $x \in X$.
 \end{enumerate}
\end{cor}

\section{Upper bounds results for extinction probabilities and fixed points}\label{sec:Moyal}

By using the submartingales of Section~\ref{subsec:martingale}, we can remove the assumption of irreducibility
from \cite[Lemma 3.3]{cf:Moyal}, a result which says that, under a mild condition, if the coordinates of $\mathbf{v}\in L_G$
are bounded away from 1, then $\mathbf{v}=\mathbf{q}(X)$.
Note that Theorem \ref{th:moyal2} (1) says that no assumptions are needed to prove that the same property holds for all $\mathbf v$ which are
extinction probability vectors.
Theorem \ref{th:moyal2} plays a key role in Section~\ref{sec:germ}.

\begin{teo}\label{th:moyal2}
Let $(X, \mu)$ be a generic MBP (not necessarily irreducible).
\begin{enumerate}
	\item If $A \subseteq X$ such that $\mathbf{q}(A) \neq \mathbf{q}(X)$, then
	$\sup_{x\in X}\mathbf q(x,A)=1$.
	\item If $\inf_{x \in X} \mathbf{q}(x,X) >0$, then for all $\mathbf z\in L_G$ such that $\mathbf z\ge \mathbf q(X)$, $\mathbf z\neq \mathbf q(X)$,  we have that $\sup_{x\in X}\mathbf z(x)=1$.
\end{enumerate}
\end{teo}

The assumption $\inf_{x \in X} \mathbf{q}(x,X) >0$, which is needed in the second part of the previous proposition, cannot be removed without replacing it by other assumptions (for instance when $X$ is finite it is not needed, see \cite[Theorem 3.4 and Corollary 3.1]{cf:BZ14-SLS}).
Indeed, without this assuption, there are examples of MBPs with an uncountable number of fixed points $\mathbf{z}$ (clearly different from $\mathbf{q}(X)$) such that $\sup_{x \in X} \mathbf{z}(x)<1$. 
Example~\ref{exmp:moyal1} shows a reducible case, while an irreducible one can be found in Example~\ref{exmp:moyal2}.

\begin{exmp}\label{exmp:moyal1}
 Let $X=\mathbb{N}$ and $\{p_n\}_{n \in \mathbb{N}}$ such that $p_n   \in(0,1)$ for all $n \in \mathbb{N}$ and $\sum_{i=0}^n (1-p_n)<+\infty$; this implies that $\prod_{i=0}^\infty p_i \in (0,1)$ and $\prod_{i=n}^\infty p_i \uparrow 1$ as $n \to +\infty$.
  Consider a MBP where a particle at $n$ has 1 child at $n+1$ with probability $p_n$ and no children with probability $1-p_n$. Clearly, if $\eta_0=\delta_0$ then, for all $n \ge 1$, either $\eta_n=\delta_n$ or $\eta_n=\mathbf{0}$.
 
 A straightforward computation shows that
 $G(\mathbf{z}|n)=1-p_n +p_n\mathbf{z}(n+1)$. 
  Moreover it is easy to show that
 $\mathbf{q}(n,X)=1-\prod_{i=n}^\infty p_i$ whence $\inf_{n \in \mathbb{N}} \mathbf{q}(n,X)=0$.
 More generally, $\mathbf{q}(A)=\mathbf{q}(X)$ if $A$ is infinite and $\mathbf{q}(A)=\mathbf{1}$ if $A$ is finite.
 
 Given $z_0 \in (1-\prod_{i=0}^\infty p_i,\, 1) =(\mathbf{q}(0,X),\, 1)$, then the recursive relation $z_{n+1}:= 1-(1-z_n)/p_n$ uniquely defines a strictly decreasing and strictly positive sequence such that $z_n > 1-\prod_{i=n}^\infty p_i$. 
 Indeed, by rewriting the recursive equality, $1-z_{n+1}=(1-z_n)/p_n>1-z_n$ for all $n \in \N$.
 The inequality $z_n > 1-\prod_{i=n}^\infty p_i$ can be proven easily by induction on $n$.
 Note that $\mathbf{z}(n):=z_n$ for all $n \in \mathbb{N}$ defines a fixed point of $G$.  Moreover $\sup_{n \in \mathbb{N}} \mathbf{z}(n)=\mathbf{z}(0)=z_0<1$.
 
 We observe, that every fixed point $\mathbf{w}$ can be constructed by interation  $\mathbf{w}(n+1):= 1-(1-\mathbf{w}(n))/p_n$ for all $n \in \mathbb{N}$ starting from $\mathbf{w}(0) \in [\mathbf{q}(0,X),\, 1]$.
 Indeed the $0$-th coordinate of a fixed point belongs to the interval $[\mathbf{q}(0,X),\, 1]$ and the iteration equality is equivalent to $G(\mathbf{w}|n)=\mathbf{w}(n)$. Thus, in this case, for every fixed point $\mathbf{w}$ (different from $\mathbf{1}$) we have $\sup_{n \in \mathbb{N}} \mathbf{w}(n)<1$.

\end{exmp}

\begin{exmp}\label{exmp:moyal2}
 Let $X=\mathbb{N}$ and $\{p_n\}_{n \in \mathbb{N}}$ as in Example~\ref{exmp:moyal1}. Moreover let $\{r_n\}$ be a sequence such that $1-p_n-r_n>0$.
  Consider a MBP where a particle at $n \ge 1$ has 1 child at $n+1$ with probability $p_n$, 1 child at $n-1$ with probability $r_n$ and no children with probability $1-p_n-r_n$. 
  Suppose that $r_0=0$, whence a particle at $0$ has 1 child at $1$ with probability $p_0$ and no children with probability $1-p_0$.
  A straightforward computation shows that
 \[
 G(\mathbf{z}|n)=
 \begin{cases}
 1-p_n-r_n +p_n\mathbf{z}(n+1)+ r_n \mathbf{z}(n-1) & n \ge 1\\
 1-p_0+p_o\mathbf{z}(1)& n=0.\\
 \end{cases}
 \]
 Clearly the generating function is smaller that the generating function of Example~\ref{exmp:moyal1}, since
 $G(\mathbf{z}|n)=
  1-p_n+p_n\mathbf{z}(n+1)- r_n (1-\mathbf{z}(n-1) )\le
  1-p_n+p_n\mathbf{z}(n+1)$;
 whence
  $\mathbf{q}(n,X)\le 1-\prod_{i=n}^\infty p_i$; again, $\inf_{n \in \mathbb{N}} \mathbf{q}(n,X)=0$.
 
 In order to prove that there are fixed points, different from $\mathbf{q}(X)$, with all coordinates smaller than $\delta$ (for some $\delta<1$), it suffices to find at least two distinct fixed points with this property.
 
 Given $z_0 \in (1-\prod_{i=0}^\infty p_i,\, 1) \subset (\mathbf{q}(0,X),\, 1)$, the recursive relation 
 \[
 z_{n+1}:=
 \begin{cases}
 1-(1-z_0)/p_0 & n=0\\
 1+(1-z_{n-1})r_n/p_n-(1-z_n)/p_n & n \ge 1\\
 \end{cases}
 \]
 uniquely defines a strictly decreasing and strictly positive sequence such that $z_n > 1-\prod_{i=n}^\infty p_i$. 
 more precisely, we prove that $z_0 \ge z_{n-1}>z_n > 1-\prod_{i=n}^\infty p_i$  by induction on $n$. The inequality $1-\prod_{i=1}^\infty p_i<z_1 < z_0$ is trivial.
 Suppose that $ 1-\prod_{i=n}^\infty p_i<z_{n}<z_{n-1}\le z_0$,
 that is $\prod_{i=n}^\infty p_i>1-z_n>1-z_{n-1} \ge 1-z_0$.
 Note that,
 $1-z_{n+1}=((1-z_n)-(1-z_{n-1})r_n)/p_n >((1-z_n)-(1-z_{n})r_n)/p_n >(1-z_n)(1-r_n)/p_n>1-z_n \ge 1-z_0$
 since, by hypothesis, $1-p_n-r_n>0$, that is, $(1-r_n)/p_n>1$. On the other hand, since
 $1-z_{n-1} > 1-z_0>0$, we have
 $1-z_{n+1}=((1-z_n)-(1-z_{n-1})r_n)/p_n < (1-z_n)/p_n<
 p_n^{-1}\prod_{i=n}^\infty p_i=\prod_{i=n+1}^\infty p_i$.
   Then $\mathbf{z}(n):=z_n$ for all $n \in \mathbb{N}$ defines a fixed point of $G$ with $\sup_{n \in \mathbb{N}} \mathbf{z}(n)=\mathbf{z}(0)=z_0<1$. 
  
  Moreover, as in Example~\ref{exmp:moyal1}, all fixed points $\mathbf{w}$ (different from $\mathbf{1}$) satisfy $\sup_{n \in \mathbb{N}} \mathbf{w}(n)<1$.
 
\end{exmp}

On may wonder when $\inf_{x \in X} \mathbf{q}(x, X) >0$ holds; the following remark gives a sufficient condition.

\begin{rem}\label{rem:conditioninf}
	If $\inf_{x\in X}\mu_x(\mathbf 0)>0$, then $\inf_{x \in X} \mathbf{z}(x) >0$ for every fixed point $\mathbf{z}$(including $\mathbf{q}(A)$ for every $A \subseteq X$). Indeed
	$\mathbf z(x)=G(\mathbf z|x)\ge\mu_x(\mathbf 0)$.
	
	Note that the existence of a nonempty subset $A$ satisfying $\inf_{x \in X} \mathbf{q}(x, A) >0$ implies the existence of $y \in X$ such that 
	$\inf_{x \in X} \mathbf{q}(x, y) >0$.
\end{rem}

It is worth noting that the existence of a positive lower bound for an extinction probability vector is a sufficient condition for the asymptotic explosion of the population. A precise statement is given by the following lemma.

\begin{lem}\label{lem:martingale2}
	Let $A \subseteq X$.
	If $\inf_{x \in X} \mathbf{q}(x, A) >0$, then
	$\pr^\eta \big (\{\sum_{x \in X} \eta_n(x)\to+\infty\} \cap \mathcal{S}(A))=\pr^\eta(\mathcal{S}(A) \big )$.
\end{lem}

\section{Germ order: survival and strong survival}\label{sec:germ}

Here we discuss survival and strong survival for MBPs under different types of stochastic dominations. We generalize the results in \cite{cf:Hut2022} by considering general MBPs instead of independent-diffusion MBPs projected on a branching process (see Section~\ref{subsec:process} for the definition).

The main result of this section is the following; this result generalizes \cite[Theorem 1.3]{cf:Hut2022}. Although our proof uses similar arguments, we stress that Theorem~\ref{th:moyal2} is the essential key which allows us to overcome the technical difficulties arising in our general case

\begin{teo}\label{th:germorder}
Let $\boldmu \gegerm \boldnu$ (with $\delta<1$) and  $A \subseteq X$.
\begin{enumerate}
	\item If $x \in X$ then $\mathbf{q}^\boldmu(x,A) \le \mathbf{q}^\boldnu(x,A) (1-\delta)+\delta$.
 \item If $x \in X$, then $\mathbf{q}^\boldnu(x,A)<\mathbf{1}$ implies 
$\mathbf{q}^\boldmu(x,A)<\mathbf{1}$.
 \item If $\sup_{x \in X} \mathbf{q}^\boldnu(x,X)<1$, then $\mathbf{q}^\boldnu(x,A)=\mathbf{q}^\boldnu(x,X)$ 
for all $x \in X$ implies
 $\mathbf{q}^\boldmu(x,A)=\mathbf{q}^\boldmu(x,X)$ 
 for all $x \in X$.
\end{enumerate}
\end{teo}

Roughly speaking, survival in $A$ for $(X, \boldnu)$ implies survival in $A$ for $(X, \boldmu)$. Moreover strong survival in $A$ for $(X, \boldnu)$ implies strong survival in $A$ for $(X, \boldmu)$.

Clearly, the germ order is not the only condition which allows to deduce strong survival for $(X, \boldmu)$ given the same behaviour for $(X,\boldnu)$. For instance if $\mu_x$ and $\nu_x$ agree outside a set $A$, then strong survival in $A$ for $(X, \boldmu)$ is equivalent to strong survival for $(X, \boldnu)$ (see \cite[Theorem 4.2]{cf:BZ2017} or \cite[Theorem 2.4]{cf:BZ2020}).

We note that the condition $\sup_{x \in X} \mathbf{q}^\boldnu(x,X)<1$ in Theorem~\ref{th:germorder} (and in Theorem~\ref{th:pgforder} below) are not necessary but it cannot be removed (see the discussion in Example~\ref{exmp:continuoustime}).  

As a warm-up, in Section~\ref{sec:proofs} we start by proving the same result under the stronger assumption $\boldmu \gepgf \boldnu$. Under this assumption, one can easily prove that $\mathbf{q}^\mu(X) \le \mathbf{q}^\nu(X)$; indeed $G_\mu(\mathbf{q}^\nu(X)) \le  G_\nu(\mathbf{q}^\nu(X))=\mathbf{q}^\nu(X)$. The following result generalizes \cite[Corollary 2.2]{cf:Hut2022}. As in the previous case, Theorem~\ref{th:moyal2} simplifies part of the proof of Theorem~\ref{th:pgforder} compared to \cite[Corollary 2.2]{cf:Hut2022}.
 
\begin{teo}\label{th:pgforder}
Let $\boldmu \gepgf \boldnu$ and  $A \subseteq X$.
\begin{enumerate}
	\item If $x \in X$, then $\mathbf{q}^\boldmu(x,A) \le \mathbf{q}^\boldnu(x,A)$; in particular $\mathbf{q}^\boldmu(x,A)=\mathbf{1}$ implies $\mathbf{q}^\boldnu(x,A)=\mathbf{1}$.
 \item If $\sup_{x \in X} \mathbf{q}^\boldnu(x,X)<1$, then
 $\mathbf{q}^\boldnu(x,A)=\mathbf{q}^\boldnu(x,X)$ for all $x \in X$ implies
 $\mathbf{q}^\boldmu(x,A)=\mathbf{q}^\boldmu(x,X)$ for all $x \in X$.
\end{enumerate}
\end{teo}

	\begin{rem}\label{rem:lowestfixedpoint1}
		One may wonder when condition $\sup_{x \in X} \mathbf{q}(x,X)<1$ is satisfied.
		We note that it holds
if and only if there exist $\mathbf{v} \in [0,1]$ and $\delta \in [0,1]$ such that $G^{(n)}(\mathbf{v}) \le \mathbf{v} \le \delta \mathbf{1}$ for some $n \ge 1$ (apply Proposition~\ref{pro:closed}). In particular if 
		\begin{equation}\label{eq:easyconditionG}
			G^{(n)}(\delta \mathbf{1}) \le \delta \mathbf{1} \textrm{ for some } n \ge 1 \textrm{ and } \delta \in [0,1],
		\end{equation}
		then 
		$\sup_{x \in X} \mathbf{q}(x,X)\le \delta$.
		An easy computation shows that $G(\delta \mathbf{1}|x)= \sum_{n \in \N} \rho_x(n) \delta^n$ where $\rho_x$ is the law of the number of children of a particle at $x$ (see the definition in Section~\ref{subsec:process}). Whence if the family of laws $\{\rho_x \colon 
			x \in X\}$ is finite and they are all supercritical,
			then equation~\eqref{eq:easyconditionG} holds.
			 Indeed, in this case, for each $x$ there exist $\delta_x \in [0,1)$ such that $\sum_{n \in \N} \rho_x(n) \delta_x^n \le \delta_x$ (choose $\delta_x=\delta_y$ if $\rho_x=\rho_y$) thus
			$G(\delta \mathbf{1}) \le \delta \mathbf{1}$ where 
			$\delta=\max_{x \in X}\delta_x$. However, condition~\eqref{eq:easyconditionG} may be satisfied even when $\rho_x$ is subcritical for some $x \in X$. 
%
\end{rem}

\begin{rem}\label{rem:lowestfixedpoint2}
	Another setting where it is easy to verify that $\sup_{x \in X} \mathbf{q}(X)<1$ is the case of $\mathcal{F}$-MBPs.
	For these MBPs, $\mathbf{q}(\cdot, X)$ assumes only a finite number of values.
	Indeed, $(X, \boldmu)$ is an $\mathcal{F}$-MBP if it can be \textit{projected} onto a MBP $(Y,\boldnu)$ where $Y$ is finite; more precisely, there exist a surjective map $g : X \mapsto Y$ such that $G_\boldmu(\mathbf{z}\circ g)=G_\boldnu(\mathbf{z})\circ g$ for all $\mathbf{z} \in [0,1]^Y$ (see \cite[Section 3.1]{cf:BZ4} for explicit computations). In \cite[Section 2.3]{cf:BZ2017} it has been shown that $\mathbf{q}^\boldmu(X)=\mathbf{q}^\boldnu(X) \circ g$ whence, if $(Y, \boldnu)$ is supercritical and irreducible, then $\sup_{x \in X} \mathbf{q}^\boldmu(x,X)=
			\sup_{x \in X} \mathbf{q}^\boldnu(g(x),Y)=\max_{y \in Y}\mathbf{q}^\boldnu(y,Y)<1$.
			A characterization of $\mathcal{F}$-MBPs with independent diffusion is given in
			\cite[Proposition 4.8]{cf:BCZ2023}.
		\end{rem}


Conditions for survival in $A$ or in $X$ for general MBPs are usually difficult to find (see for instance \cite[Theorem 4.1]{cf:Z1} and \cite[Theorems 3.1 and 3.2]{cf:BZ14-SLS}). 
Theorem~\ref{th:germorder} (2) and Proposition~\ref{pro:germindepdiff} together provide a powerful tool to prove survival for 
MBPs with independent diffusion. 
Indeed suppose that $(X, \boldnu)$ is a MBP with independent diffusion and survives in $A$. Then, any other MBP $(X, \boldmu)$
with independent diffusion, with the same matrix $P$, such that
 condition (2) of Proposition~\ref{pro:germindepdiff} holds, 
survives in $A$, no matter how inhomogeneous the offspring distributions of $(X, \boldmu)$ are.
This applies for instance to the case of global survival ($A=X$).
If $(X,\boldnu)$ is an $\mathcal{F}$-MBP, then it survives globally if and only if the Perron-Frobenius eigenvalue of a finite matrix is strictly larger than 1 (see \cite[Theorem 4.3]{cf:Z1}, \cite[Theorem 3.1]{cf:BZ14-SLS} and \cite[Section 2.4]{cf:BZ14-SLS}). An $\mathcal{F}$-MBP with independent diffusion is completely described by \cite[Proposition 4.8]{cf:BCZ2023}.
Thus, we may be able to identify when $(X,\boldnu)$ survives globally and,
by Theorem~\ref{th:germorder} (2), claim that $(X, \boldmu)$ survives globally as well, even if 
$(X,\boldmu)$ can be fairly inhomogeneous.

We observe that Proposition~\ref{pro:germindepdiff} 
gives a condition for MBPs with independent diffusion equivalent to the germ order.
An application of Proposition~\ref{pro:germindepdiff} is the following: suppose that $(X, \boldnu)$ is an irreducible and quasi-transitive MBP with independent diffusion (see for instance \cite[Section 2.4]{cf:BZ14-SLS}). Consider another MBP with independent diffusion $(X,\boldmu)$  such that condition (2) of Proposition~\ref{pro:germindepdiff} holds. If there exists
$x \in X$ such that $\mathbf{q}^\boldnu(x,x)<1$, then for every nonempty set $A \subseteq X$ we have $\mathbf{q}^\boldmu(w,X)=\mathbf{q}^\boldmu(w,A)< 1$ for all $w \in X$. Indeed, if $y \in A \subseteq X$, according to \cite[Corollary 3.2]{cf:BZ14-SLS}, $\mathbf{q}^\boldnu(x,x)<1$ implies $\mathbf{q}^\boldnu(w,X) \le \mathbf{q}^\boldnu(x, A) \le \mathbf{q}^\boldnu(w,y)=\mathbf{q}^\boldnu(w,X)<1$ for all $w \in X$. Moreover, since a quasi-transitive MBP is an $\mathcal{F}$-MBP, by Remark~\ref{rem:lowestfixedpoint2} we have $\sup_{w \in X} \mathbf{q}^\boldnu(w, X)<1$.
Proposition~\ref{pro:germindepdiff} and Theorem~\ref{th:germorder} yields the claim.

The following example shows that if we have two MBPs with independent 
diffusion and the offspring distribution is geometric, then the pgf and germ ordering are both equivalent to the coordinate-wise
ordering of the first moment matrices.

\begin{exmp}\label{exmp:continuoustime}

		If $\boldmu$ satisfies equation~\eqref{eq:particular1}, then 
$G_\boldmu({\mathbf{z}}|x)=\sum_{n \in \N} \rho_x(n) (P{\mathbf{z}}(x))^n$
(see Section~\ref{subsec:genfun}).
If, in particular, $\rho_x(n)=\frac{1}{1+\bar \rho_x} (\frac{\bar \rho_x}{1+\bar \rho_x} )^n$ 
(as in the discrete-time counterpart of a continuous-time MBP,  see \cite[Section 2.2]{cf:Z1} for details), then
the previous expression becomes $G_\boldmu({\mathbf{z}}|x)=(1+\bar \rho_x P(\mathbf{1}-\mathbf{z})(x))^{-1}$ 
or, in a more compact way, 
\begin{equation}\label{eq:Gcontinuous}
	G_\boldmu({\mathbf{z}})= \frac{\mathbf{1}}{\mathbf{1}+M_\boldmu(\mathbf{1}-{\mathbf{z}})}
\end{equation}
where $M_\boldmu$ is the first-moment matrix and $M_\boldmu \mathbf{v}(x)=\bar \rho_x P\mathbf{v}(x)$.
Suppose that $\boldmu$ and $\boldnu$ satisfy equation~\eqref{eq:particular1} (possibly with different matrices $P_\boldmu$ and $P_\boldnu$); let $M_\boldmu$ and $M_\boldnu$ be the first moment matrices of $\boldmu$ and $\boldnu$ respectively. By using equation~\eqref{eq:Gcontinuous}, the following assertions are equivalent: (1) $M_\boldmu \ge M_\boldnu$ (with the usual natural partial order), (2) $M_\boldmu \mathbf{v} \ge M_\boldnu \mathbf{v}$ for all $\mathbf{v} \in [0,1]^X$, (3)
$\boldmu \gepgf \boldnu$, (4) $\boldmu \gegerm \boldnu$. Therefore, Theorem~\ref{th:pgforder}(1) applies, to ensure that extinction
of the $\boldmu$-process implies extinction of the $\boldnu$-process.	
In order to apply Theorem~\ref{th:pgforder}(2) (strong survival with $\boldnu$ implies strong survival with $\boldmu$),
we need $\sup_{x \in X} \mathbf{q}^\nu(x, X)<1$. 
According to Remark~\ref{rem:lowestfixedpoint1}, a sufficient condition for $\sup_{x \in X} \mathbf{q}^\nu(x, X)<1$ is the existence of $\delta <1$ such that $G_\nu(\delta \ident |x) \le \delta$ for all $x \in X$: if $G$ is as in equation~\eqref{eq:Gcontinuous} this condition is equivalent to $\inf_{x \in X} \bar \rho_x^\boldnu >1$. 

The above equivalences, along with Theorem~\ref{th:pgforder}(2), seem to suggest a monotonicity of strong local survival with respect to the first moment matrix, that is, that increasing the first moments of a process with strong survival, produces a new process which still has 
strong survival;
 however we know from \cite[Example 4.2]{cf:BZ14-SLS} that this is false. 
 Indeed  \cite[Example 4.2]{cf:BZ14-SLS} shows that one can find three ordered first moment matrices $M_\boldnu \le M_\boldmu \le M_{\bm{\sigma}}$ such that, for all finite $A \subset X$,  $\mathbf{q}^\boldnu(A)=\mathbf{q}^\boldnu(X)$, $\mathbf{q}^{\bm{\sigma}}(A)=\mathbf{q}^{\bm{\sigma}}(X)$ but $\mathbf{q}^\boldnu(A)>\mathbf{q}^\boldnu(X)$ (that is, strong survival with $\bm\nu$ and $\bm\sigma$, but not with $\bm\mu$).
 
 In particular,  this
 proves that the condition $\sup_{x \in X} \mathbf{q}^\boldnu(x,X)<1$ in Theorems~\ref{th:germorder}~(3)~and~\ref{th:pgforder}~(2) cannot be removed. Indeed in  \cite[Example 4.2]{cf:BZ14-SLS}, $\sup_{x \in X} \mathbf{q}^\boldnu(x,X)=1$. 
 Hence, without the condition that  $\sup_{x \in X} \mathbf{q}^\boldnu(x,X)< 1$, even if there is strong survival in $A$ for the $\boldnu$-process, there are measures $\bm{\sigma}\gegerm\boldnu$ and $\boldmu\gegerm\boldnu$ such that there is strong survival in $A$ for the $\bm{\sigma}$-process and not for the $\boldmu$-process.

\end{exmp}

We close this section with an application of Theorem~\ref{th:germorder} to survival in a sequence of subsets.

\begin{cor}\label{cor:maxdispl}
	Let $\boldmu \gegerm \boldnu$ and consider a sequence $\{A_n\}_{n\ \in \mathbb{N}}$ of subsets of $X$.
		\begin{enumerate}
			\item If $x \in X$ and $\pr^x_\boldnu \big (\limsup_{n \to +\infty} \{\sum_{y \in A_n}\eta_n(y)>0\}\big )>0$,
 then $\pr^x_\boldmu \big (\limsup_{n \to +\infty} \{\sum_{y \in A_n}\eta_n(y)>0\}\big )>0$.
			\item If $\sup_{x \in X} \mathbf{q}^\boldnu(x, X)<1$ and $\pr^x_\boldnu \big (\liminf_{n \to +\infty} \{\sum_{y \in A_n}\eta_n(y)=0\}\big )=\mathbf{q}^\boldnu(x,X)$ for all $x \in X$,
 then $\pr^x_\boldmu\big (\liminf_{n \to +\infty} \{\sum_{y \in A_n}\eta_n(y)=0\}\big )=\mathbf{q}^\boldmu(x,X)$  for all $x \in X$.
			\item If $\sup_{x \in X} \mathbf{q}^\boldnu(x, X)<1$ and $\pr^x_\boldnu \big (\limsup_{n \to +\infty} \{\sum_{y \in A_n}\eta_n(y)>0\}|\mathcal{S}(X) \big )=1$  for all $x \in X$,
 then $\pr^x_\boldmu\big (\limsup_{n \to +\infty} \{\sum_{y \in A_n}\eta_n(y)>0\}|\mathcal{S}(X) \big )=1 \ \forall x \in X$.
		\end{enumerate}
\end{cor}

\begin{exmp}\label{exmp:maximal}
	As an application of Corollary~\ref{cor:maxdispl} consider a metric $d$ on $X$; for instance, $d$ could be the natural metric induced by a connected graph structure on $X$. Fix $x_0 \in X$ and define the \textit{maximal and minimal displacements} as $M_n:=\ident_{\mathcal{S}(X)} \cdot \max \{d(x_0,y)\colon y \in X, \, \eta_n(y)>0\}$ $m_n:=\ident_{\mathcal{S}(X)} \cdot \min \{d(x_0,y)\colon y \in X, \, \eta_n(y)>0\}$. If $\boldmu \gegerm \boldnu$ then, given $\alpha >0$ and $f : \mathbb{N} \mapsto (0,+\infty$),
	\[
	\begin{split}
		\limsup_{n \to +\infty} M_n /f(n) \le \alpha,\, \pr^{x_0}_\boldmu\text{-a.s.~} &\Longrightarrow 
		\limsup_{n \to +\infty} M_n /f(n) \le \alpha, \, \pr^{x_0}_\boldnu\text{-a.s.} \\
		\liminf_{n \to +\infty} m_n /f(n) \ge \alpha,\, \pr^{x_0}_\boldmu\text{-a.s.~} &\Longrightarrow 
		\liminf_{n \to +\infty} m_n /f(n) \ge \alpha, \, \pr^{x_0}_\boldnu\text{-a.s.} \\
	\end{split}
	\]
	The details can be found in Section~\ref{sec:proofs}.	
\end{exmp}	

We observe that, in principle, the main results of this section can be extended to MBPs in varying environment; these are MBPs where $\boldmu=\{\mu_{x,n}\}_{x \in X, n \in \mathbb{N}}$ 
and the reproduction law of a particle at $x$ at time $n$ is $\mu_{x,n}$. 
Such processes admit a  space-time counterpart (as in the proof of Lemma~\ref{lem:germorder2}, see also \cite{cf:BRZ16}) which is a MBP in a fixed environment. Such an extension, however, goes beyond the purpose of this paper.


\section{Proofs}\label{sec:proofs}

\begin{proof}[Proof of Proposition~\ref{pro:Gtopology}]
	
	\leavevmode
	\begin{enumerate}
		\item It is easy: see \cite[Sections 2 and 3]{cf:BZ2} for the details.    
		\item It is enough to prove that $\mathbf{z} \mapsto G_\boldmu(\mathbf{z}|y)=\sum_{f \in S_X} \mu_y(f) \prod_{x \in X} \mathbf{z}(x)^{f(x)}$ is continuous with respect to the pointwise convergence topology for each $y \in Y$. To this aim, note that $\sup_{\mathbf{z} \in [0,1]^X} \big |\mu_y(f) \prod_{x \in X} \mathbf{z}(x)^{f(x)} \big |=\mu_y(f)$ and $\mathbf{z} \mapsto \prod_{x \in X} \mathbf{z}(x)^{f(x)}$ is continuous with respect to the pointwise convergence topology for every $f \in S_X$. Since $\sum_{f \in S_X} \mu_y(f)=1<+\infty$ then $\sum_{f \in S_X} \mu_y(f) \prod_{x \in X} \mathbf{z}(x)^{f(x)}$ converges to $G_\boldmu(\mathbf{z}|y)$ uniformly with respect to $\mathbf{z} \in [0,1]^X$, therefore $G_\boldmu(\cdot|y)$ is continuous.
		\item 
		Let us note that the tightness of $\{\rho_y\}_{y \in Y}$ means that for every $\varepsilon >0$ there exists $n=n(\varepsilon) \in \mathbb{N}$ such that $\mu_y(f\colon |f| \le n) > 1-\varepsilon $ for all $y \in Y$.
		We start by proving that $\big | \prod_{x \in X} \mathbf{z}(x)^{f(x)}-\prod_{x \in X} \mathbf{v}(x)^{f(x)} \big | \le \min(1, |f| \cdot \|\mathbf{z}-\mathbf{v}\|_\infty)$. The inequality $\big | \prod_{x \in X} \mathbf{z}(x)^{f(x)}-\prod_{x \in X} \mathbf{v}(x)^{f(x)} \big | \le 1$ follows trivially from the fact that $|t-s| \le \max(|t|,|s|)$ for $t,s \ge 0$. As for the second inequality, observe that if $t_i, s_i \in [0,1]$ for all $i=1, \ldots,n$ then
		\[
		\begin{split}
			\prod_{i=1}^n t_i-\prod_{i=1}^n s_i =&
			\prod_{i=1}^n t_i-s_n \prod_{i=1}^{n-1} t_i+ s_n \prod_{i=1}^{n-1} t_i+ \cdots +\prod_{i=1}^{j} t_i\prod_{i=j+1}^{n} s_i - \prod_{i=1}^{j-1} t_i\prod_{i=j}^{n} s_i\\ &+ \cdots + t_1 \prod_{i=2}^n s_i -  \prod_{i=1}^n s_i
		\end{split}
		\]
		whence
		\[
		\begin{split}
			\Big |\prod_{i=1}^n t_i-\prod_{i=1}^n s_i\Big | \le &
			\Big |\prod_{i=1}^{n-1} t_i\Big | \cdot |t_n-s_n|+ \cdots +
			\Big |\prod_{i=1}^{j-1} t_i\prod_{i=j+1}^{n} s_i|\cdot |t_j-s_j|\\ &+ \cdots + \Big |\prod_{i=2}^n s_i\Big |\cdot |t_1-s_1| \le n \max_{i=1, \ldots, n} |t_i-s_i|.
		\end{split}
		\]
		Let us fix $\varepsilon >0$ and let $n=n(\varepsilon/2)$ (coming from the tightness). For every $\mathbf{z}, \mathbf{v}$ such that $\|\mathbf{z}-\mathbf{v}\|_\infty \le \delta:=\varepsilon/(2n)$ and for every $y \in Y$ we have
		\[
		\begin{split}
			\big |G_\boldmu(\mathbf{z}|y)-G_\boldmu(\mathbf{v}|y) \big | &\le 
			\sum_{f \in S_X\colon |f| \le n} \mu_y(f) \Big |\prod_{x \in X}  \mathbf{z}(x)^{f(x)} - \prod_{x \in X}\mathbf{v}(x)^{f(x)} \Big |\\&+
			\sum_{f \in S_X\colon |f| >n} \mu_y(f) \Big |\prod_{x \in X}  \mathbf{z}(x)^{f(x)} - \prod_{x \in X}\mathbf{v}(x)^{f(x)} \Big |\\
			&\le \sum_{f \in S_X\colon |f| \le n}  \mu_y(f) |f|\cdot \|\mathbf{z}-\mathbf{v}\|_\infty+ \sum_{f \in S_X\colon |f| >n} \mu_y(f)
			\\
			& \le n\frac{\varepsilon}{2n}+\frac{\varepsilon}{2}=\varepsilon.
		\end{split}
		\]
		Hence $\|\mathbf{z}-\mathbf{v}\|_\infty \le \delta$ implies $\big \|G_\boldmu(\mathbf{z})-G_\boldmu(\mathbf{v}) \big \|_\infty \le \varepsilon$.
	\end{enumerate}
\end{proof}

We prove now that the binary relation $\gegerm$ is a partial order on the space of all generating functions from $[0,1]^X$ to $[0,1]^Y$. To this aim we need a lemma.	

\begin{lem}\label{lem:uniqueness}
	Let  $G(\mathbf{z})$ be a holomorphic function defined on $D^n$ where $D$ is the closed unit ball in $\mathbb C$.
	Suppose that $G$ vanishes on $[\delta,1]^n$ for some $0\le\delta<1$. Then $G$ vanishes on $D^n$.
\end{lem}

\begin{proof}
	We first prove the statement for $n=2$. 
	Let $\mathbf{z}=\big (\mathbf{z}(1),\mathbf{z}(2) \big )$ where $\mathbf{z}(1)\in [\delta,1]$ (meaning that the imaginary part of $\mathbf{z}(1)$ is 0 and the real part belongs to the interval) and let $G_{\mathbf{z}(1)}(w):=G(\mathbf{z}(1),w)$
	for all $w\in D$. By hypothesis, $G_{\mathbf{z}(1)}$ is a holomorphic function of one complex variable which vanishes on the
	real interval $[\delta,1]$. Since this interval has at least one limit point in $D$, then it is well-known that  $G_{\mathbf{z}(1)}(w)=0$
	for all $w\in D$ (see for instance \cite[Theorem 10.18]{cf:Rudin}). This proves that $G$ vanishes on $[\delta,1] \times D$.
	
	Now fix $s\in D$ and define $G_s(\xi):=G(\xi,s)$ for all $\xi\in D$. This is again a holomorphic function of one complex variable which vanishes on the real interval $[\delta,1]$. 
	By the same argument as before, $G_s(\xi)=0$ for all $\xi\in D$.
	This means that $G(\mathbf{z})$ vanishes in $D^2$. The statement for a general $n>2$ follows easily by induction.
\end{proof}

\begin{proof}[Proof of Proposition~\ref{pro:germordertrueorder}]
	The relation $\gegerm$ is clearly reflexive and transitive. Let us prove it is antisymmetric. Suppose that $\boldmu \gegerm \boldnu$ and $\boldnu \gegerm \boldmu$, that is, there exists $\delta<1$ such that for all $\mathbf{z} \in [\delta,1]^X$, then $G_\boldmu(\mathbf{z})=G_\boldnu(\mathbf{z})$; we prove that $\boldmu=\boldnu$ (this is equivalent to $G_\boldmu=G_\boldnu$ as discussed in Section~\ref{subsec:genfun}).
	It is enough to prove that $\mu_y=\nu_y$ (or equivalently that $G_\boldmu(\cdot|y)=G_\boldnu(\cdot|y)$) for every fixed $y \in Y$.
	
	To this aim, note that equation~\eqref{eq:genfun} defines a continuous function $G$ on $D^X$ where $D:=\{z \in \mathbb{C} \colon |z| \le 1\}$ is the closed disk of radius 1 in the complex plane.
	Whence when $X$ is finite, for every fixed $y \in X$, the generating function $G(\cdot|y)$ can be seen as a holomorphic function of several variables.
	In this case the result follows from Lemma~\ref{lem:uniqueness}; indeed since $G_\boldmu(\cdot|x)-G_\boldnu(\cdot|x)$ vanishes on $[\delta,1]^X$, then, by Lemma~\ref{lem:uniqueness}, it vanishes on $D^X$ whence $\mu_y(f)-\nu_y(f)=0$ for every $f \in S_X$.
	
	Now let $X$ be infinite; given a subset $W \subseteq X$ 	define $V(W):=\{\mathbf{z} \in [0,1]^X \colon \mathbf{z}(x)=1, \, \forall x \in X \setminus W\}$ and let $\pi:V(W) \mapsto [0,1]^W$ be the bijective map defined 
	as $\pi(\mathbf{z}):=\mathbf{z}|_W$ (the restriction of $\mathbf{z}$ to $W$). Given $f \in S_X$ define $ \langle f \rangle _W:=\{g \in S_X \colon g|_W=f|_W\}$ the set of functions extending the restriction of $f$ to $W$; moreover define $S_X(W):=\{f \in S_X\colon \{f>0\} \subseteq W\}\}$ 
	the set of finitely supported functions on $X$ whose support is in $W$. Clearly, since $x \to f(x) \ident{(x \in W)}$ is a map in $S_X(W)$ and $ \langle f \rangle _W= \langle f(\cdot)\ident(\cdot \in W) \rangle _W$, then the map $f \mapsto  \langle f \rangle _W$ is a bijection 
	from $S_X(W)$ onto $\{ \langle f \rangle _W \colon f \in S_X\}$. 
	Roughly speaking, $ \langle f \rangle _W$ are equivalence classes containing exactly one function $g \in S_X(W)$ and since every $g \in S_X(W)$ belongs to a class, there is a one to one correspondence between $S_X(W)$ and $\{ \langle f \rangle _W \colon f \in S_X\}$.
	
	We observe now that $G_\boldmu(\cdot|x)|_{V(W)}$ and $G_\boldnu(\cdot|x))|_{V(W)}$ can be seen as functions defined on $[0,1]^W$, indeed 
	$G_\boldmu(\pi^{-1}(\cdot)|y)|_{V(W)}=G_\boldmu(\pi^{-1}(\cdot)|y)$ and $\pi^{-1}$ is a bijection from $[0,1]^W$ onto $V(W)$ (and the same holds for $\boldnu$). More precisely
	\[
	G_\boldmu(\pi^{-1}(\mathbf{z})|y)=
	\sum_{f \in S_X(W)} \mu_x( \langle f \rangle _W)\prod_{w \in W} \mathbf{z}(w)^{f(w)}, 
	\] 
	and an analogous expression holds for $G_\boldnu$.
	Suppose that $W$ is finite; since $G_\boldmu(\pi^{-1}(\cdot)|y)=G_\boldnu(\pi^{-1}(\cdot)|y)$ on $[\delta,1]^W$ the same equality holds on $V(W)$ (by Lemma~\ref{lem:uniqueness}). This implies easily that $\mu_y( \langle f \rangle _W)=\nu_y( \langle f \rangle _W)$ for every $f \in S_X(W)$ or, equivalently, for every $f \in S_X$. Consider now a fixed sequence of finite subsets of $X$, say $\{W_n\}_{n \in \mathbb{N}}$, such that $W_n \subseteq W_{n+1}$ and $\bigcup_{n \in \mathbb{N}} W_n=X$.
	Then, for all $f \in S_X$ we have $ \langle f \rangle _{W_{n+1}}\subseteq  \langle f \rangle _{W_n}$ and $\bigcap_{n \in \mathbb{N}}( \langle f \rangle _{W_n})=\{f\}$, 
	therefore 
	\[
	\mu_y(f)=\lim_{n \to +\infty} \mu_y( \langle f \rangle _{W_n})=\lim_{n \to +\infty} \nu_y( \langle f \rangle _{W_n})=\nu_y(f).
	\]
\end{proof}

\begin{proof}[Proof of Proposition~\ref{pro:germindepdiff}]
	\noindent $(1) \Rightarrow (2)$. 
	Using the hypothesis and the expression for $\phi_y$, we get that
	for every $t \in [\delta,1]$ and for all $x \in X$, since $t \mathbf{1} \in [\delta,1]^X$ then
	\[
	\phi_y^\boldmu(t)
	=G_\boldmu(t \mathbf{1}|y) \le G_\boldnu(t \mathbf{1}|y)
	=\phi_x^\boldnu(t).
	\]
	
	\noindent $(2) \Rightarrow (1)$. 
	Recall that, for multinomial families, the generating functions are
	$G_\boldmu(\mathbf{z}|y)=\phi_y^\boldmu(P\mathbf{z}(y))$ and
	$G_\boldnu(\mathbf{z}|y)=\phi_y^\boldnu(P\mathbf{z}(y))$.
	We observe that the map $\mathbf{z} \mapsto P\mathbf{z}$
	is nondecreasing and continuous from $[0,1]^X$ into $[0,1]^Y$; in
	particular, if $\mathbf{z} \in [\delta,1]^X$ for some $\delta<1$, then $P\mathbf{z} \in [\delta,1]^Y$. Indeed
	$P\, t \mathbf{1}=t \mathbf{1}$ therefore $\delta \mathbf{1}=P\, \delta \mathbf{1} \le P \mathbf{z} \le P \mathbf{1}=\mathbf{1}$.
	Take $\mathbf{z} \in [\delta,1]^X$; then for all $y \in Y$
	\[
	G_\boldmu(\mathbf{z}|y)=\phi_y^\boldmu(P\mathbf{z}(y)) \le
	\phi_y^\boldnu(P\mathbf{z}(y))=
	G_\boldnu(\mathbf{z}|y)
	\]
	where we used the inequality $\phi_y^\boldmu(t) \le \phi_y^\boldnu(t)$ for $t=P\mathbf{z}(y) \in [\delta,1]$ (due to the monotonicity of $P$).
\end{proof}

Since Lemma~\ref{lem:martingale} and Proposition~\ref{pro:martingale} hold for every initial condition $\eta$, in order to avoid a cumbersome notation, in the proofs we use $\pr$ and $\E$ instead of $\pr^\eta$ and $\E^\eta$.

\begin{proof}[Proof of Lemma~\ref{lem:martingale}]
	Let $k=1$. We write the explicit expression of $\eta_{m+1}$ as a function of $\eta_m$ and identify $\eta_m(\omega)$ with a function 
	$h\in S_X$. Then
	\begin{equation}\label{eq:z-eta}
		\begin{split}
			\E[\mathbf z^{\eta_{m+1}}|\mathcal F_m] & =\E\Big[\prod_{x\in X}\mathbf z(x)^{\sum_{y\in X}\sum_{i=1}^{\eta_m(y)}f_{i,m,y}(x)}|\mathcal F_m\Big]\\
			&
			= \sum_{h\in S_X} \ident{(\eta_m=h)}
			\E\Big[\prod_{x\in X}\mathbf z(x)^{\sum_{y\in X}\sum_{i=1}^{h(y)}f_{i,m,y}(x)}|\mathcal F_m\Big],\ \pr\textrm{-a.s.}
		\end{split}
	\end{equation}
	where in the last equality we used the fact that $\eta_m$ is $\mathcal F_m$-measurable.
	Using indepedence of $f_{i,m,y}$ and $\mathcal F_m$, we get
	\[
	\E\Big[\prod_{x\in X}\mathbf z(x)^{\sum_{y\in X}\sum_{i=1}^{h(y)}f_{i,m,y}(x)}|\mathcal F_m\Big]
	=\E\Big[\prod_{x\in X}\prod_{y\in X}\prod_{i=1}^{h(y)}\mathbf z(x)^{f_{i,m,y}(x)}\Big],\ \pr\textrm{-a.s.}
	\]
	Now, since  $\{f_{i,m,y}(x)\}_{i,m \in \mathbb{N}, y \in X}$ is a family of independent random variables, this expectation can be written as
	(by definition of $G$)
	\[
	\prod_{y\in X}\prod_{i=1}^{h(y)}\E\Big[\prod_{x\in X}\mathbf z(x)^{f_{i,m,y}(x)}\Big]=
	\prod_{y\in X}\prod_{i=1}^{h(y)}
	G(\mathbf z|y).
	\]
	Thus \eqref{eq:z-eta} becomes
	\begin{equation}\label{eq:z-eta2}
		\begin{split}
			\E[\mathbf z^{\eta_{m+1}}|\mathcal F_m] & 
			= \sum_{h\in S_X} \ident{(\eta_m=h)}
			\prod_{y\in X}\prod_{i=1}^{h(y)}
			G(\mathbf z|y) = \sum_{h\in S_X} \ident{(\eta_m=h)}
			\prod_{y\in X}
			G(\mathbf z|y)^{h(y)}\\
			&=\sum_{h\in S_X} \ident{(\eta_m=h)}
			\prod_{y\in X}
			G(\mathbf z|y)^{\eta_m(y)}=G(\mathbf z)^{\eta_m}, \ \pr\textrm{-a.s.}
		\end{split}
	\end{equation}
	which proves the claim for $k=1$.
	
	The claim is proven by induction on $k$. Indeed
	\[
	\begin{split}
		\E[\mathbf z^{\eta_{m+k}}|\mathcal F_m] & 
		= \E\Big[ \E\big[\mathbf z^{\eta_{m+k}}|\mathcal F_{m+k-1}\big] |\mathcal F_m\Big]
		=  \E\Big[ G(\mathbf z)^{\eta_{m+k-1}} |\mathcal F_m\Big] \\
		&
		=  \left(G^{(k-1)}(G(\mathbf z))\right)^{\eta_{m}}=(G^{(k)}(\mathbf z))^{\eta_m}, \ \pr\textrm{-a.s.}
	\end{split}
	\]
	where in the last line we used the induction hypothesis and the definition of $G^{(k)}$.
\end{proof}

\begin{proof}[Proof of Proposition~\ref{pro:martingale}]
	The two inequalities come from Lemma~\ref{lem:martingale} and they hold for every initial state of the process; in particular if $\mathbf{z} \in F_G$, 
	then $\{\mathbf{z}^{\eta_n}\}_{n \in \mathbb{N}}$ is a martingale.
	
	Note that $\{\mathbf{z}^{\eta_n}\}_{n \in \mathbb{N}}$ is uniformly bounded by the constant function 1, whence it is a uniformly integrable family. It is well-known that a supermartigale or a submartingale bounded in $L^1(\pr)$ converges a.s., whence
%
$\mathbf{z}^{\eta_n} \to W_\mathbf{z}
\pr
\text{-a.s.}~\text{and in }L^p(\pr
)
$,
where the $L^p(\pr
	)$ convergence comes from the a.s.~convergence and the Bounded Convergence Theorem. The $L^1(\pr
	)$-convergence and the fact that $\E
	[\mathbf{z}^{\eta_m} | \mathcal{F}_n] \ge \mathbf{z}^{\eta_n}$ (resp.~$\E
	[\mathbf{z}^{\eta_m} | \mathcal{F}_n] \le \mathbf{z}^{\eta_n}$) for all $m \ge n$
	implies
	$\E
	[W_\mathbf{z} | \mathcal{F}_n]\ge \mathbf{z}^{\eta_n}$ (resp.~$\E
	[W_\mathbf{z} | \mathcal{F}_n]\le\mathbf{z}^{\eta_n}$).	
\end{proof}

\begin{proof}[Proof of Corollary~\ref{cor:martingale}]
	Note that $\E^\eta[\ident_{\mathcal{E}(A)}| \mathcal{F}_n]=\mathbf{q}(A)^{\eta_n}$.
	Indeed, it is enough to note that, for every sequence $\{f_i\}_{i=1}^n$ in $S_X$, the Markov property implies
	\[
	\pr^\eta(\mathcal{E}(A)|\eta_1=f_1, \ldots, \eta_n=f_n)=\pr^x(\mathcal{E}(A)|\eta_n=f_n)=
	\mathbf{q}(A)^{f_n}.
	\]
	By Proposition~\ref{pro:martingale}, $\{\mathbf{q}(A)^{\eta_n}\}_{n \in \mathbb{N}}$
	is a martingale and, by \cite[Theorem 14.2]{cf:Williams} (or \cite[Theorem 9.4.8]{cf:Chung}, since $ \ident_{\mathcal{E}(A)} \in L^p(\pr^x)$, then
	$\E^\eta[\ident_{\mathcal{E}(A)}| \mathcal{F}_n] \to \E^\eta[\ident_{\mathcal{E}(A)}| \mathcal{F}_\infty]=\ident_{\mathcal{E}(A)}$,
	$\pr^\eta$-a.s.~and in $L^p(\pr^\eta)$
	for all $p \ge 1$.
\end{proof}

 \begin{proof}[Proof of Corollary~\ref{cor:martingale2}]
	The equivalence between (2) and (3) follows from the identity
	$\pr^\eta=\ast_{x \in X} \ast_{i=1}^{\eta(x)} \pr^x$ where $\ast$ is the usual convolution product of measures. Let us see the details.
	
	\noindent $(2) \Rightarrow (3)$. There is nothing to prove.
	
	\noindent $(3) \Rightarrow (2)$. Consider, on a suitable probability space, a family
	$\{\{\eta_n^{i,x}\}_n\}_{i \in \mathbb{N}, x \in X}$ of independent MBPs such that $\eta_0^{i,x}=\delta_x$. By the superimposition property (or by equation~\eqref{eq:evolBRW}) we have that $\eta_n:=\sum_{x \in X} \sum_{i=1}^{\eta(x)} \eta_n^{i,x}$ is a MBP with initial condition $\eta$. Whence the law of $\{\eta_n^{i,x}\}_n$ is $\pr^x$ and the law of $\{\eta_n\}_{n}$ is $\pr^\eta$. Clearly $\mathbf{z}^{\eta_n}= \prod_{x \in X} \prod_{i=1}^{\eta(x)} \mathbf{z}^{\eta_n^{i,x}}$.
	Since $\mathbf{z}^{\eta_n^{i,x}} \to W^{i,x}_\mathbf{z}$ a.s., then $\mathbf{z}^{\eta_n} \to \prod_{x \in X} \prod_{i=1}^{\eta(x)} W^{i,x}_\mathbf{z}$ a.s. A similar argument holds for $\mathbf{v}$ instead of $\mathbf{z}$.
	By hypothesis $W^{i,x}_\mathbf{z} \ge W^{i,x}_\mathbf{v}$ a.s.~whence $\prod_{x \in X} \prod_{i=1}^{\eta(x)} W^{i,x}_\mathbf{z} \ge \prod_{x \in X} \prod_{i=1}^{\eta(x)} W^{i,x}_\mathbf{z}$ a.s.
	
	\noindent $(1) \Rightarrow (2)$. 
	Suppose that $\mathbf{z} \ge  \mathbf{v}$, then $\mathbf{z}^{\eta_n} \ge  \mathbf{v}^{\eta_n}$ $\pr^\eta$-a.s.~thus by taking the limit as $n \to +\infty$,  Proposition~\ref{pro:martingale} yields  $\pr^\eta(W_\mathbf{z} \ge
	W_\mathbf{v})=1$.
	
	\noindent $(2) \Rightarrow (1)$.
	Finally, suppose that $\pr^x(W_\mathbf{z} \ge
	W_\mathbf{v})=1$ for every $x \in X$. Then $\mathbf{z}(x)=\E^x[W_\mathbf{z}] \ge \E^x[W_\mathbf{v}] = \mathbf{v}(x)$ for all $x \in X$.
\end{proof}

In order to prove Theorem~\ref{th:moyal2}(2), we need Lemma~\ref{lem:martingale2},
thus we proceed with its proof first.\\

	\begin{proof}[Proof of Lemma~\ref{lem:martingale2}]
		Let $\inf_{x \in X} \mathbf{q}(x, A)=:\alpha>0$. If $\alpha=1$, then there is nothing to prove, since $\pr^\eta(\mathcal{S}(A))=0$.
		If $\alpha<1$, from Corollary~\ref{cor:martingale} we have that, $\pr^\eta$-a.s.~on $\mathcal{S}(A)$, 
		\[
		0=\lim_{n \to +\infty} \mathbf{q}(A)^{\eta_n} \ge \lim_{n \to +\infty}\alpha^{\sum_{x \in X} \eta_n(x)}.
		\]
		Thus $\pr^\eta(\{\lim_{n \to +\infty}\alpha^{\sum_{x \in X} \eta_n(x)}=0\}\cap \mathcal{S}(A))=\pr^\eta(\mathcal{S}(A))$,
		which implies the claim.
	\end{proof}

\begin{proof}[Proof of Theorem~\ref{th:moyal2}]\leavevmode
	\begin{enumerate}
	\item 	The statement is 
	\cite[Corollary 4.2]{cf:BBHZ}.
	\item Assume now that $\inf_{x \in X} \mathbf{q}(x,X) >0$. By hypothesis, $\mathbf z(x)\ge\mathbf q(x,X)$ for all $x\in X$ and there exists  $x_0$ such that $\mathbf z(x_0)>\mathbf q(x_0,X)$.
	Suppose by contradiction that $\mathbf z(x)\le 1-\varepsilon$ for all $x\in X$, for some $\varepsilon >0$.
	Let $W_\mathbf{z}:=\lim_{n\to+\infty} \mathbf z^{\eta_n}$. 
	On $\mathcal E(X)$ we have $W_\mathbf{z}=1$ 
	(see discussion after Proposition \ref{pro:martingale}).
	By Lemma \ref{lem:martingale2}, on $\mathcal S(X)$, $W_\mathbf{z}=\lim_{n\to+\infty} \mathbf z^{\eta_n}\le\lim_{n\to+\infty}(1-\varepsilon)^{\sum_{x\in X}\eta_n(x)}=0$,
	$\pr^{x_0}$-a.s. Whence $W_\mathbf{z}=\ident_\mathcal E(X)$, $\pr^{x_0}$-a.s.
	Thus 
	\[
	\mathbf q(x_0,X)<\mathbf z(x_0) \le \E^{x_0}[W_\mathbf{z}]=\E^{x_0}[\ident_\mathcal E(X)]=\mathbf q(x_0,X)
	\]
	which is a contradiction.
		
	\end{enumerate}
	
\end{proof}

Define $L(A):=\sum_{x \in A, n \in \N} \eta_n(x)$ the total number of visits in $A$; clearly, $\mathbf{q}(x,A)=\pr^x(L(A)<+\infty)$ for all $x \in X$. Moreover let $L_n(A):=\sum_{x \in A, i \le n} \eta_i(x)$ be the  number of visits in $A$ before time $n$; clearly $L_n(A) \uparrow L(A)$ as $n \to +\infty$.
Before proving Theorem~\ref{th:germorder}, as a warm-up, we prove Theorem~\ref{th:pgforder}; to this aim we need a preparatory lemma.

\begin{lem}\label{lem:pgforder}
	Let $\boldmu \gepgf \boldnu$ and  $A \subseteq X$. Then $\E_\boldmu^x[\exp(-tL(A))] \le
	\E_\boldnu^x[\exp(-tL(A))]$ for all $t \ge 0$, $x \in X$.
\end{lem}

\begin{proof}
	We prove by induction on $n$ that
	$\E_\boldmu^x[\exp(-tL_n(A))] \le
	\E_\boldnu^x[\exp(-tL_n(A))]$ for all $t \ge 0$, $x \in X$. the claim follows from the Bounded Convergence Theorem.
	
	If $n=0$, then, for all $x \in X$ and $t \in [0,+\infty)$, 
	$\E_\boldmu^x[\exp(-tL_0(A))]= \exp(-t) \ident_A(x)+\ident_{A^\complement}(x)=
	\E_\boldnu^x[\exp(-tL_0(A))]$.
	
	Let $n \ge 0$ and suppose that $\E_\boldmu^x[\exp(-tL_{n}(A))] \le
	\E_\boldnu^x[\exp(-tL_{n}(A))]$ for all $t \ge 0$, $x \in X$.
	We have
	\[
	\begin{split}
		\E_\boldmu^x[\exp(-tL_{n+1}(A))] &= (
		\exp(-t) \ident_A(x)+\ident_{A^\complement}(x))
		\sum_{f \in S_X} \mu_x(f) \prod_{y \in X} \E^y_\boldmu[\exp(-tL_n(A))]^{f(y)}\\
		&=(
		\exp(-t) \ident_A(x)+\ident_{A^\complement}(x))G_{\boldmu}\big (\E_\boldmu^{(\cdot)}[\exp(-tL_n(A))] \big |x \big )\\
		_\textrm{\tiny{(induction) }}&\le (
		\exp(-t) \ident_A(x)+\ident_{A^\complement}(x))
		G_{\boldmu}\big (\E_\boldnu^{(\cdot)}[\exp(-tL_n(A))] \big |x \big )\\
		_\textrm{\tiny{(pgf order) }}&\le (
		\exp(-t) \ident_A(x)+\ident_{A^\complement}(x))G_{\boldnu}\big (\E_\boldnu^{(\cdot)}[\exp(-tL_n(A))] \big |x \big )\\
		&=\E_\boldnu^x[\exp(-tL_{n+1}(A))].
	\end{split}
	\]
	
\end{proof}

\begin{proof}[Proof of Theorem~\ref{th:pgforder}]\leavevmode
	
	\begin{enumerate}
		\item 
		By the Bounded Convergence Theorem and Lemma~\ref{lem:pgforder}, for all $x \in X$ and $A \subseteq X$ we have
		\[
		\begin{split}
			\mathbf{q}^\boldmu(x,A)&=\pr^x_\boldmu (L(A)<+\infty)=
			\lim_{t\to 0^+} \E_\boldmu^x[\exp(-tL(A))] \\
			&\le \lim_{t\to 0^+}
			\E_\boldnu^x[\exp(-tL(A))] =
			\pr^x_\boldnu (L(A)<+\infty)=\mathbf{q}^\boldnu(x,A).
		\end{split}
		\]
		\item We know from (1) and from the hypotheses that, 
		\[
		\mathbf{q}^\boldmu(A)
		\le \mathbf{q}^\boldnu(A) =
		\mathbf{q}^\boldnu(X).
		\]
		Then
		$\sup_{x \in X} \mathbf{q}^\boldmu(x,A)<1$ which, according to Theorem~\ref{th:moyal2}  
		, implies
		$\mathbf{q}^\boldmu(A) =
		\mathbf{q}^\boldmu(X)$.
	\end{enumerate}
\end{proof}

We can prove now Theorem~\ref{th:germorder}. We need two preparatory lemmas. The first one is the analogous of \cite[Lemma 2.3]{cf:Hut2022} and the proof is on the same line.  As usual, $\vee$ and $\wedge$ denote the maximum and the minimum respectively.

\begin{lem}\label{lem:germorder}
	Let $\boldmu \gegerm \boldnu$ and  $A \subseteq X$. If $\delta<1$ is the same as in the definition of $\gegerm$, then for all $t \in [\delta, 1]$ and all $x \in X$,
	\[
	\E_\boldnu^x[t^{\ident{(L(A)>0)}}] \ge t \vee \E_\boldmu^x[t^{L(A)}].
	\]
\end{lem}

\begin{proof}
	From Definition~\ref{def:ordering}, for every $\mathbf{z} \in [\delta,1]^X$ (that is, for every $\mathbf{z} \in [0,1]^X$ such that $\delta \mathbf{1} \le \mathbf{z} \le \mathbf{1}$) we have
	$G_\boldmu(\mathbf{z}) \le G_\boldnu(\mathbf{z})$. If $t=1$ there is nothing to prove. Let us fix $t \in (\delta,1)$ (the case $t=\delta$ follows by taking the limit). Clearly $G_\boldmu(\mathbf{z}) \le G_\boldnu(\mathbf{z})$ for all $\mathbf{z} \in [t,1]^X$.
	
	The strategy of the proof is to find $\mathbf{v}_\infty, \mathbf{w}_\infty	\in [t,1]^X$ such that $\E_\boldnu^x[t^{\ident{(L(A)>0)}}] \ge \mathbf{v}_\infty(x) \ge \mathbf{w}_\infty(x) \ge t \vee \E_\boldmu^x[t^{L(A)}]$ for all $x \in X$.
	To this aim define $I_\boldmu, I_\boldnu : [t,1]^X \mapsto [t,1]^X$ as follows
	\begin{equation}\label{eq:Imu}
		\begin{split}
			I_\boldmu \mathbf{z}(x)&:=
			\big (t \vee t^{\ident(x \in A)}G_\boldmu(\mathbf{z}|x)\big ) \wedge \mathbf{z}(x)\\
			&=t \vee \big (t^{\ident(x \in A)}G_\boldmu(\mathbf{z}|x) \wedge \mathbf{z}(x)\big )
			=
			\begin{cases}
				t & x \in A\\
				t \vee \big (G_\boldmu(\mathbf{z}|x) \wedge \mathbf{z}(x)\big )& x \not \in A\\	
			\end{cases}
		\end{split}
	\end{equation}
	and $I_\boldnu$ is defined analogously by using $G_\boldnu$ instead of $G_\boldmu$. It is easy to show that 
	$I_\boldmu, I_\boldnu$ are nondecreasing, continuous functions on $[t,1]^X$. Moreover, for all $\mathbf{z} \in [t,1]^X$ we have
	$t \mathbf{1} \le I_\boldmu \mathbf{z} \le I_\boldnu \mathbf{z} \le \mathbf{z}$.
	Define recursively
	\[
	\begin{cases}
		\mathbf{v}_0(x)=\mathbf{w}_0(x):=t^{\ident(x \in A)}, & \forall x \in X,\\
		\mathbf{v}_{n+1}:=I_\boldnu \mathbf{v}_n, & \forall n \in \mathbb{N},\\
		\mathbf{w}_{n+1}:=I_\boldmu \mathbf{w}_n, & \forall n \in \mathbb{N},\\
	\end{cases}
	\]
	whence $\{\mathbf{w}_n\}_{n \in \mathbb{N}}$ and $\{\mathbf{v}_n\}_{n \in \mathbb{N}}$ are nonincreasing sequences in $[t,1]^X$ such that
	$t \mathbf{1} \le \mathbf{w}_n \le \mathbf{v}_n \le \mathbf{z}$, therefore $\mathbf{v}_n \downarrow \mathbf{v}_\infty$, $\mathbf{w}_n \downarrow \mathbf{w}_\infty$	and
	$t \mathbf{1} \le \mathbf{w}_\infty \le \mathbf{v}_\infty \le \mathbf{z}$. By the same arguments of Proposition~\ref{pro:closed},
	we have $I_\boldnu \mathbf{v}_\infty=\mathbf{v}_\infty$ and $I_\boldmu \mathbf{w}_\infty=\mathbf{w}_\infty$.
	We prove now, by induction on $n \in \mathbb{N}$, that $\mathbf{w}_n(x) \ge t \vee \E_\boldmu^x[t^{L_n(A)}]$ for all $n \in \mathbb{N}$ which, in turn, implies $\mathbf{w}_\infty(x) \ge t \vee \E_\boldmu^x[t^{L(A)}]$. If $n=0$, then $\mathbf{w}_0(x)=t^{\ident(x \in A)} \ge t \vee \E_\boldmu^x[t^{L_0(A)}]$ since $\ident(x \in A)=L_0(A)$. 
	Suppose that the inequality holds for $n \in \mathbb{N}$, then, by using that the MBP is a stationary Markov process and that the set of descendants of different particles belonging to a fixed generation are independent, we have for all $x \in X$
	\[
	\begin{split}
		\E_\boldmu^x[t^{L_{n+1}(A)}]&=	\E_\boldmu^x \big [\E_\boldmu^x[t^{L_{n+1}(A)} \big | \mathcal{F}_1]\big ]	
		= t^{\ident(x \in A)} \sum_{f \in S_X} \mu_x(f) \prod_{y \in X}\E_\boldmu^y[t^{L_n(A)}]^{f(y)}\\
		&= t^{\ident(x \in A)} G_\boldmu(\E_\boldmu^{(\cdot)}[t^{L_n(A)}]|x) \le 
		t^{\ident(x \in A)} G_\boldmu(\mathbf{w}_n|x)
	\end{split}
	\]
	(where $\E_\boldmu^{(\cdot)}[t^{L_n(A)}]$ represents the vector $y \mapsto \E_\boldmu^y[t^{L_n(A)}]$). Note that in the last inequality we used the induction hypothesis and the fact that $G_\boldmu$ is nondecreasing. Clearly $\E_\boldmu^x[t^{L_{n+1}(A)}] \le \E_\boldmu^x[t^{L_{n}(A)}]
	\le \mathbf{w}_n(x)$, thus
	\[
	t \vee \E_\boldmu^x[t^{L_{n+1}(A)}] \le
	t \vee \big (\mathbf{w}_n(x) \wedge t^{\ident(x \in A)} G_\boldmu(\mathbf{w}_n|x)  \big ) =I_\boldmu \mathbf{w}_n=
	\mathbf{w}_{n+1}.
	\]
	
	Now we prove that $\E_\boldnu^x[t^{\ident{(L(A)>0)}}] \ge \mathbf{v}_\infty(x)$ for all $x \in X$. Let us define $D:=\{x \in X \colon \mathbf{v}_\infty(x)=t\}$; clearly, since $t\le \mathbf{v}_\infty(x)\le t^{\ident(x\in A)}$ for all $x \in X$, then $D \supseteq A$. Define recursively
	\[
	\begin{cases}
		\mathbf{h}_0(x):=t^{\ident(x \in D)}, & \forall x \in X\\
		\mathbf{h}_{n+1}:=I_\boldnu \mathbf{h}_n & \forall n \in \mathbb{N}.
	\end{cases}
	\]
	The sequence $\{\mathbf{h}_n\}_{n \in \mathbb{N}}$ is nondecreasing therefore $\mathbf{h}_n \downarrow \mathbf{h}_\infty$ for some $\mathbf{h}_\infty \in [t,1]^X$. Moreover, since $I_\boldnu \mathbf{v}_\infty= \mathbf{v}_\infty \le \mathbf{h}_0$, then $t \le \mathbf{v}_\infty(x) \le \mathbf{h}_\infty(x) \le t^{\ident(x \in D)}$; thus $\mathbf{h}_n(x)=t$ for all $x \in D$. On the other hand, if $x \not \in D$, then, by definition of $D$, $ t < \mathbf{v}_\infty(x) \le \mathbf{h}_n(x)$ for all $n \in \mathbb{N}$ and $G_\boldnu(\mathbf{h}_n|x) \ge G_\boldnu(\mathbf{v}_\infty|x)=\mathbf{v}_\infty (x)> t$ for all $n \in \mathbb{N}$. Therefore, by using equation~\eqref{eq:Imu},
	\[
	\mathbf{h}_{n+1}(x) =
	\begin{cases}
		t & x \in D\\
		G_\boldnu(\mathbf{h}_n|x) \wedge \mathbf{h}_n(x)& x \not \in D.\\	
	\end{cases}
	\]
	Define $E_n(D)$ as the number of particles in $D$ by time $n$ with no ancestors in $D$ and let $E(D):=\lim_{n \to +\infty} E_n(D)$ (note that $E_{n+1}(D) \ge E_n(D)$). If, for instance, $x \in D$, then $E_n(D)=1$ for all $n \in \mathbb{N}$. We want to prove that $\mathbf{h}_n(x)=\E_\boldnu^x[t^{E_n(D)}]$ for all $x \in X$ which, according to the Bounded Convergence Theorem, implies $\mathbf{h}_\infty(x)=\E_\boldnu^x[t^{E(D)}]$ for all $x \in X$.
	To this aim note that $L(A)>0$ implies $E(D) \ge 1$, therefore $\E_\boldnu^x[t^{E(D)}] \le \E_\boldnu^x[t^{\ident(L(A)>0)}]$ for all $x \in X$. Define
	$\mathbf{\widetilde h}_n(x):=\E_\boldnu^x[t^{E_n(D)}]$ for all $x \in X$.
	By using again the fact that the MBP is a stationary Markov process and that the progenies of different particles are independent, we see that the (nonincreasing) sequence $\{\mathbf{\widetilde h}_n(x)\}_{n \in \mathbb{N}}$ satisfies the following recursive equation  for all $x \in X$
	\[
	\begin{split}
		\mathbf{\widetilde h}_{n+1}(x)=	\E_\boldnu^x[t^{E_{n+1}(D)}]&=
		\begin{cases}
			t & x \in D\\
			\E_\boldnu^x \big [	\E_\boldnu^x[t^{E_{n+1}(D)}|\mathcal{F}_1] \big ]= (\spadesuit) & x \not \in D
		\end{cases}\\
		(\spadesuit)&=
		\sum_{f \in S_X} \nu_x(f) \prod_{y \in X}\E_\boldnu^y[t^{E_{n}(D)}]^{f(y)}=G_\boldnu(\mathbf{\widetilde h}_n|x) = \mathbf{\widetilde h}_n(x) \wedge G_\boldnu(\mathbf{\widetilde h}_n|x)
	\end{split}
	\]
	where, in the last equality, we used the fact that, by definition, $\mathbf{\widetilde h}_{n+1}(x) \le \mathbf{\widetilde h}_{n}(x)$ for all $x \in X$, which implies   $G_\boldnu(\mathbf{\widetilde h}_n|x) = \mathbf{\widetilde h}_{n+1}(x)\le \mathbf{\widetilde h}_{n}(x)$ for all $x \not \in D$.
	We observe that $\mathbf{\widetilde h}_{0}=\mathbf{h}_{0}$ since $\E_\boldnu^x[E_0(D)]=\ident(x \in D)$ for all $x \in X$; moreover the sequences $\{\mathbf{\widetilde h}_n(x)\}_{n \in \mathbb{N}}$ and $\{\mathbf{h}_n(x)\}_{n \in \mathbb{N}}$ satisfy the same recursive equation, hence
	$\mathbf{\widetilde h}_n=\mathbf{h}_n$ for all $n \in \mathbb{N}$. This yields
	\[
	\E_\boldnu^x[t^{\ident(L(A)>0)}] \ge \E_\boldnu^x[t^{E(D)}]= \lim_{n \to +\infty} \mathbf{\widetilde h}_n(x)= \lim_{n \to +\infty} \mathbf{h}_n(x)= \mathbf{h}_\infty  \ge \mathbf{v}_\infty.
	\]
\end{proof}

\begin{lem}\label{lem:germorder2}
	Let $\boldmu \gegerm \boldnu$ and  $A \subseteq X$.  If $\delta<1$ is the same as in the definition of $\gegerm$, then for all $t \in [\delta,1)$ all $x \in X$, 
	\[
	\mathbf{q}^\boldnu(x,A) \ge \frac{\mathbf{q}^\boldmu(x,A) \vee t-t}{1-t}.
	\]
\end{lem}

\begin{proof}
	In order to prove this lemma (by using Lemma~\ref{lem:germorder}) we define an auxiliary space-time version of the MBP (as in \cite[Lemma 2.3]{cf:Hut2022}). More precisely, given a MBP $\{\eta_n\}_{n \in \mathbb{N}}$ on $X$ we denote by $\{\eta^{st}_n\}_{n \in \mathbb{N}}$ a MBP on  $X \times \mathbb{N}$ that we call \textit{space-time version} of the original process and which is defined by
	$\eta_n^{st}(x,m):=\eta_n(x)\delta(n,m)$ (where $\delta(n,m)=1$ if $n=m$ and $0$ otherwise). Roughly speaking, the particles in $x$ at time $n$ in the original MBP, are now placed in $(x,n)$ at time $n$ in the st-MBP. The space-time version of $\boldmu$, say $ \boldmu^{st}$ is defined as follows, $\forall g \in S_{X \times \mathbb{N}}$ and $\forall (x,n) \in X \times \mathbb{N}$,
	\[
	\mu_{(x,n)}^{st}(g)=
	\begin{cases}
		\mu_x(f) & \text{if }g=f \otimes \delta_{n+1}\\
		0 & \text{otherwise}
	\end{cases}
	\]
	where $\big (f \otimes \delta_{i} \big ) (y,j):=f(y) \delta(i,j)$ for all $(y,j) \in X \times \mathbb{N}$.
	
	Elementary computations show that for all $\mathbf{z} \in [0,1]^{X \times \mathbb{N}}$, 
	$G_{\boldmu^{st}}(\mathbf{z}|(x,n))= G_\boldmu(\mathbf{z}(\cdot, n+1)|x)$ and $G_{\boldnu^{st}}(\mathbf{z}|(x,n))= G_\boldnu(\mathbf{z}(\cdot, n+1)|x)$. If $\boldmu \gegerm \boldnu$, then $\boldmu^{st} \gegerm \boldnu^{st}$. Indeed if $\mathbf{z} \in [\delta,1]^{X \times \mathbb{N}} $ (where $\delta<1$), then $\mathbf{z}(\cdot,n) \in [\delta,1]^X$ for all $n \in \mathbb{N}$ whence
	\[
	G_{\boldmu^{st}}(\mathbf{z}|(x,n))= G_\boldmu(\mathbf{z}(\cdot, n+1)|x) \le  G_\boldnu(\mathbf{z}(\cdot, n+1)|x)=G_{\boldnu^{st}}(\mathbf{z}|(x,n))
	\]
	for all $(x,n) \in X \times \mathbb{N}$.
	
	Moreover $A \subseteq X$ is visited infinitely often by $(X, \boldmu)$ (resp.~$(X, \boldnu)$) if and only if $A \times \mathbb{N}$ is visited infinitely often by $(X \times \mathbb{N}, \boldmu^{st})$ (resp.~$(X \times \mathbb{N}, \boldnu^{st})$). 
	In particular $\mathbf{q}^\boldmu(x,A)=\mathbf{q}^{\boldmu^{st}}((x,n),A \times \mathbb{N})$ and 
	$\mathbf{q}^\boldnu(x,A)=\mathbf{q}^{\boldnu^{st}}((x,n),A \times \mathbb{N})$ for all $(x,n) \in X \times \mathbb{N}$, $A \subseteq X$.
	Thus, it suffices to prove the lemma for the space-time version of the MBP.
	
	To avoid a cumbersome notation, for the rest of the proof we write $\boldmu$ and $\boldnu$ instead of $\boldmu^{st}$ and $\boldnu^{st}$ respectively. Moreover we use $\pr_\boldmu^{x,n}$ and $\pr_\boldnu^{x,n}$ to denote the laws of the space-time processes starting from $(x,n)$. Given $A \subseteq X \times \mathbb{N}$, we define $A_k:=A \cap \big (X \times [k, +\infty) \big )$. We observe that 
	\[
	L(A)=+\infty 
	\Longleftrightarrow L(A_k)>0, \, \forall k \in \mathbb{N}
	\Longleftrightarrow L(A_k)>0, \, \text{for infinitely many } k \in \mathbb{N}
	\]
	since $\{L(A_{k+1})>0\} \subseteq \{L(A_{k})>0\}$ and at every fixed time the number of particles is finite.
	Whence $\{L(A)=+\infty\}=\bigcap_{k \in \mathbb{N}} \{L(A_{k})>0\}$ and $\{L(A)<+\infty\}=\liminf_{k \in \mathbb{N}} \{L(A_{k})=0\}$. This implies $\ident(L(A_k)>0) \downarrow \ident(L(A)=+\infty)$.
	Note that $L(A)=+\infty$ implies $L(A_k)=+\infty$ for all $k \in \mathbb{N}$ while $L(A)<+\infty$ implies $L(A_k)=0$ eventually as $k \to +\infty$.
	
	We apply Lemma~\ref{lem:germorder} to $A_k$ and, for every fixed $(x,n) \in X \times \mathbb{N}$, we obtain
	\begin{equation}\label{eq:germ1}
		\E_\boldnu^{x,n}[t^{\ident(L(A_k)>0)}] \ge t \wedge \E_\boldmu^{x,n}[t^{L(A_k)}].
	\end{equation}
	According to the Monotone Convergence Theorem 
	\begin{equation}\label{eq:germ2}
		\lim_{k \to +\infty} \E_\boldnu^{x,n}[t^{\ident(L(A_k)>0)}]=\E_\boldnu^{x,n}[t^{\ident(L(A)=+\infty)}]=t(1-\mathbf{q}^{\boldnu}((x,n),A))+\mathbf{q}^{\boldnu}((x,n),A). 
	\end{equation}
	According to the Bounded Convergence Theorem , if $t<1$
	\begin{equation}\label{eq:germ3}
		\lim_{k \to +\infty} \E_\boldmu^{x,n}[t^{L(A_k)}]=\E_\boldmu^{x,n}[\ident(L(A)<+\infty)]=\mathbf{q}^{\boldmu}((x,n),A).
	\end{equation}
	By using equations~\eqref{eq:germ1}, \eqref{eq:germ2} and \eqref{eq:germ3} we obtain
	\[
	t(1-\mathbf{q}^{\boldnu}((x,n),A))+\mathbf{q}^{\boldnu}((x,n),A) \ge t \wedge \mathbf{q}^{\boldmu}((x,n),A)
	\]
	which yields the result.
\end{proof}

\begin{proof}[Proof of Theorem~\ref{th:germorder}]\leavevmode
	\begin{enumerate}
		\item By taking $t=\delta$ in Lemma~\ref{lem:germorder2} we have
		\[
		\mathbf{q}^\boldnu(x,A) \ge \frac{\mathbf{q}^\boldmu(x,A) \vee \delta-\delta}{1-\delta} \ge \frac{\mathbf{q}^\boldmu(x,A)-\delta}{1-\delta}
		\]
		which yields the claim.
		\item 
		Fix $x \in X$ and suppose that 
		$\mathbf{q}^\boldmu(x,A)=1$. Then by Lemma~\ref{lem:germorder2}, if we choose $t \in (\delta,1)$ we have
		\[
		\mathbf{q}^\boldnu(x,A) \ge \frac{\mathbf{q}^\boldmu(x,A) \vee t-t}{1-t}=
		\frac{1-t}{1-t}=1.
		\]
		\item Suppose that $\mathbf{q}^\boldnu(A)=\mathbf{q}^\boldnu(X)$ and that $\sup_{x \in X} \mathbf{q}^\boldnu(x,X) <1$. Then, by Lemma~\ref{lem:germorder2},
		\[
		\begin{split}
			1 &> \sup_{x \in X} \mathbf{q}^\boldnu(x,X)
			= \sup_{x \in X} \mathbf{q}^\boldnu(x,A) 
			\ge \frac{\sup_{x \in X} \mathbf{q}^\boldmu(x,A) \vee t-t}{1-t}
			\ge 	\frac{\sup_{x \in X} \mathbf{q}^\boldmu(x,A) -t}{1-t}	
		\end{split}
		\]
		which is equivalent to $\sup_{x \in X} \mathbf{q}^\boldmu(x,A) <1$. According to Theorem~\ref{th:moyal2} the last inequality implies $\mathbf{q}^\boldmu(A)=\mathbf{q}^\boldmu(X)$.		
	\end{enumerate}
\end{proof}

%

\begin{proof}[Proof of Corollary~\ref{cor:maxdispl}]
	Consider, as in the proof of Lemma~\ref{lem:germorder2},
	the space-time version $\{\eta^{st}_n\}_{n \in \mathbb{N}}$ of the process. Clearly
	\[
	\limsup_{n \to +\infty} \Big \{\sum_{y \in A_n} \eta_n(y)>0 \Big \}=\mathcal{S}^{st} \Big (\bigcup_{n \in \mathbb{N}} \big (A_n \times \{n\} \big )\Big )
	\]
	where $\mathcal{S}^{st}(\cdot)$ is the survival event of the space-time process. Recall that, for all $A \subseteq X$, $\mathbf{q}^\boldmu(x, A)=\mathbf{q}^{\boldmu^{st}}((x,n), A \times \mathbb{N})$ and 
	$\mathbf{q}^\boldnu(x, A)=\mathbf{q}^{\boldnu^{st}}((x,n), A \times \mathbb{N})$ for all $(x,n) \in X \times \mathbb{N}$. 
			$(1)$ and $(2)$ follows from Theorem~\ref{th:germorder} applied to the space-time process. $(3)$ follow from $(2)$ by noting that 
	$\limsup_{n \to +\infty} \{\sum_{y \in A_n}\eta_n(y)>0\} \subseteq \mathcal{S}(X)$.
\end{proof}

\begin{proof}[Details on Example~\ref{exmp:maximal}]
	We note that 
	\[
	\begin{split}
	\pr^{x_0}_\boldmu (\limsup_{n \to +\infty} M_n /f(n) \le \alpha) =1 &\Longleftrightarrow 
	\pr^{x_0}_\boldmu
	\Big ( \limsup_{n \to +\infty}\Big \{\sum_{y \colon d(x_0,y) \ge (\alpha+\varepsilon)f(n)} \eta_n(y) >0 \Big \}
	\Big ) =0, \ \forall \varepsilon >0	\\
	\pr^{x_0}_\boldmu (\liminf_{n \to +\infty} m_n /f(n) \le \alpha) =1 &\Longleftrightarrow 
	\pr^{x_0}_\boldmu
	\Big ( \limsup_{n \to +\infty}\Big \{\sum_{y \colon d(x_0,y) \le (\alpha-\varepsilon)f(n)} \eta_n(y) >0 \Big \}
	\Big ) =0, \ \forall \varepsilon >0	\\
	\end{split}
	\]
	and similar equalities hold for $\boldnu$. The result follows by applying Corollary~\ref{cor:maxdispl}
	to $A^{\varepsilon, +}_n:=\{y \in X \colon d(x_0,y) \ge (\alpha +\varepsilon)f(n)\}$ and $A^{\varepsilon, -}_n:=\{y \in X \colon d(x_0,y) \le (\alpha -\varepsilon)f(n)\}$..
\end{proof}

\section{Appendix: product of metric spaces}

In this appendix we show how the product of metric spaces can be endowed with a finite metric which generates the pointwise convergence
topology. We also address separability and completeness. We note that $\mathbb R^X$ can be endowed with a finite metric which turns
it into a Polish space.

\begin{lem}\label{lem:finitemetric}
  Consider a metric space $(Y,d)$ and a function $f \in L^1([0, +\infty))$ such that $f$ is non increasing a.e.~and $\int_0^\varepsilon f(t) \diff t>0$ for all $\varepsilon >0$. 
  Then $d_1(x,y):= \int_0^{d(x,y)} f(t) \diff t$ for all $x,y \in Y$ defines a finite metric which generates the same topology.
\end{lem}

\begin{proof}
Note that $f$ is a.s.~nonnegative and $\int_0^a f(t) \diff t=0$ if and only if $a=0$. Whence $d_1(x,y) \ge 0$ for all $x,t \in Y$ and the equality holds if and only if $d(x,y)=0$, that is, $x=y$. 
As for the triangle inequality
\[
\begin{split}
 d_1(x,z)+d_1(z,y)&=\int_0^{d(x,z)} f(t) \diff t+\int_0^{d(z,y)} f(t) \diff t \\
 &{\ge} 
 \int_0^{d(x,z)} f(t) \diff t+\int_0^{d(z,y)} f(t+d(x,z)) \diff t \\
 & = \int_0^{d(x,z)} f(t) \diff t+\int_{d(x,z)}^{d(x,z)+d(z,y)} f(t) \diff t\\
& = \int_{0}^{d(x,z)+d(z,y)} f(t) \diff t \ge \int_0^{d(x,z)} f(t) \diff t
= d_1(x,y).
\end{split}
\]
Finally $d_1(x,y) \le \|f\|_1:=\int_0^\infty f(t) \diff t <+\infty$ for all $x, y \in Y$.
 
Let us prove that the topology is the same.
On the one hand $B(x, r)=B_1(x, \int_0^r f(t) \diff t)$ for all $r>0$. On the other hand, $\varepsilon \mapsto \int_0^\varepsilon f(t) \diff t$ is right continuous in 0, whence for every $r>0$ there exists $\varepsilon>0$ such that $0<\int_0^\varepsilon f(t) \diff t=:r_1<r$, that is,
$B(x,r) \supseteq B_1(x, r_1)$.
\end{proof}

An example is given by $f:=\ident_{[0,M]}$ which gives
$d_1(x,y)=\min(d(x,y),M)$, where $M>0$. Since the topology is the same, if the original metric space is separable (resp.~complete) the same hold for the new one. The advantage of a finite metric is clear in the following lemma.

We suppose that $\{(Y_n,d_n)\}_{n \in J}$ is a countable (finite or infinite) sequence of finite metric spaces where $\sup_{x,y \in Y_n} d_n(x,y) = M_n<+\infty$.

\begin{pro}\label{pro:productspace}
 Let $\{\alpha_n\}_{n \in J}$ a sequence of positive real numbers such that $\sum_{n \in J} \alpha_n M_n <+\infty$. Consider the product space $\prod_{n \in J} Y_n$ endowed with the product topology (the \emph{pointwise convergence topology}).
 Then 
 \begin{equation}\label{eq:productdistance}
 d(\mathbf{z}, \mathbf{v}):=
 \sum_{n \in J} \alpha_n d_n(\mathbf{z}(n),\mathbf{v}(n) )
 \end{equation}
is a finite metric on $\prod_{n \in J} Y_n$ which generates the pointwise convergence topology.
\end{pro}

\begin{proof}
 The defining properties of a metric for $d$ follow easily from the corresponding properties for every $d_n$.
 
 We denote by $\mathbf{y}$ an element  of the product space and $\mathbf{y}(i)$ is called the $i$th coordinate. 
 Recall that the product topology of $\prod_{n \in J} Y_n$ is the smallest topology containing the basic open sets $<E_n>_{n \in S}:=\{\mathbf{y} \in \prod_{n \in J} Y_n\colon \mathbf{y}(i) \in E_i, \, \forall i \in S\}$, where $S \subseteq J$ is finite and $E_i$ is an open subset of $Y_i$ (for every $i \in S$).

 Suppose that $A \subseteq \prod_{n \in J} Y_n$ is an open set and $\mathbf{y} \in A$. Then, by definition of product topology, there exist a finite $S \subseteq J$ and a collection of open sets $\{E_i\}_{i \in S}$ such that 
 $\mathbf{y} \in <E_i>_{i \in S} \subseteq A$. Since $\mathbf{y}(i) \in E_i$ and $E_i$ is open, then for every $i \in S$, there exists $r_i>0$ such that
 $\mathbf{y}(i) \in B_n(\mathbf{y}(i) , r_i) \subseteq E_i$. Define $\beta:=\min\{\alpha_i r_i\colon i \in S\}$; it is easy to show that
 $B(\mathbf{y}, \beta) \subseteq <E_i>_{i \in S}$. Indeed, if
 $\mathbf{z} \in B(\mathbf{y}, \beta)$, then $d(\mathbf{y}, \mathbf{z}) \le \beta$ which implies $d_i(\mathbf{y}(i), \mathbf{z}(i)) \le \alpha^{-1} \beta \le r_i$ for all $i \in S$. Whence, $\mathbf{z} \in <E_i>_{i \in S}$.
 
 Conversely, consider $B(\mathbf{y}, r)$.  We show that there exist a finite $S \subseteq J$ and a collection of open sets $\{E_i\}_{i \in S}$ such that 
 $\mathbf{y} \in <E_i>_{i \in S}  \subseteq B(\mathbf{y}, r)$. Since 
 $\sum_{n \in J} \alpha_n M_n <+\infty$, there exists a finite $S \subseteq J$ such that
 $\sum_{n \in J\setminus S} \alpha_n M_n < r/2$. Define $r_n:=r/(2 \alpha_n \# S)$ for every $n \in S$ where $\# S<+\infty$ is the cardinality of $S$. If $\mathbf{z}$ is such that 
 $d_n(\mathbf{y}(n), \mathbf{z}(n)) \le  r_n$ for all $n \in S$. Then
 \[
 \begin{split}
  d(\mathbf{y}, \mathbf{z}) &=
  \sum_{n \in S} d_n(\mathbf{y}(n), \mathbf{z}(n)) \alpha_n +
  \sum_{n \in J \setminus S} d_n(\mathbf{y}(n), \mathbf{z}(n)) \alpha_n\\
  & \le \sum_{n \in S} r_n \alpha_n +
  \sum_{n \in J \setminus S} M_n \alpha_n < r/2+r/2=r
  \end{split}
 \]
Whence, if $E_n:=B_n(\mathbf{y}(n),r_n)$ for all $n \in S$, then
$\mathbf{y} \in <E_i>_{i \in S}  \subseteq B(\mathbf{y}, r)$.
 
\end{proof}

The following lemma is elementary but we include it for the sake of completeness. It generalizes to metric spaces a well-known result on \emph{total convergence} in normed space.

\begin{lem}\label{lem:totalconvergence}
 Let $(Y,d)$ be a metric space. The space is complete if and only if every sequence $\{y_i\}_{i \in \mathbb{N}}$ such that $\sum_{i \in \mathbb{N}} d(y_i,y_{i+1})<+\infty$ converges.
\end{lem}
\begin{proof}
 Suppose that $(Y,d)$ is complete. By using the triangle inequality, $d(y_n,y_m) \le \sum_{i=n}^{m-1} d(y_i,y_{i+1})$ for all $n <m$, if
$\sum_{i \in \mathbb{N}} d(y_i,y_{i+1})<+\infty$, then $\{y_i\}_{i \in \mathbb{N}}$ is a Cauchy sequence, whence it is convergent.
 
 Conversely, suppose that every sequence $\{y_i\}_{i \in \mathbb{N}}$ such that $\sum_{i \in \mathbb{N}} d(y_i,y_{i+1})<+\infty$ converges. Let $\{y_i\}_{i \in \mathbb{N}}$ be a Cauchy sequence. Define $n_i:=\min\{n \in \N \colon d(y_j,y_m) \le 1/2^{i+1}, \, \forall j,m \ge n\}$. By construction
 $\sum_{i \in \mathbb{N}} d(y_{n_i},y_{n_{i+1}}) \le \sum_{i \in \mathbb{N}} 1/2^{i+1} =1<+\infty$, whence the subsequence $\{y_{n_i}\}_{i \in \N}$ converges to some $z \in Y$. Let $\varepsilon >0$ and ${i_\varepsilon}$ such that $1/2^{{i_\varepsilon}} \le \varepsilon$. By continuity, $d(y_{n_{{i_\varepsilon}}}, z)\le 1/2^{{i_\varepsilon}+1}$ and $d( y_{n_{{i_\varepsilon}}}, y_n)
 \le 1/2^{{i_\varepsilon}+1}$ for every $n \ge n_{{i_\varepsilon}}$. Thus,
for all $n \ge n_{{i_\varepsilon}}$, 
$d(y_n,z) \le d(y_{n_{{i_\varepsilon}}}, z)+d( y_{n_{{i_\varepsilon}}}, y_n) \le 1/2^{{i_\varepsilon}} \le \varepsilon$ and
this proves that the space is complete.
 
\end{proof}

\begin{rem}\label{rem:productspace}
 It is known, see for instance \cite{cf:Hewitt}, that if every $Y_i$ is separable and the cardinality of $J$ is at most $2^{\aleph_0}$,
 then $\prod_{n \in J} Y_n$ is separable. The converse is  trivial.
 
 Moreover, by using Lemma~\ref{lem:totalconvergence} it is easy to show that 
  every finite metric space $(Y_i, d_i)$ is complete if and only if $\prod_{n \in J} Y_n$ is complete with the distance~\eqref{eq:productdistance}. Indeed, suppose that every finite metric space $(Y_i, d_i)$ is complete. Since $d(\mathbf{y}, \mathbf{z})/\alpha_i \ge d_i(\mathbf{y}(i),\mathbf{z}(i))$ for every $i \in J$, if
  $\sum_{n \in \N} d(\mathbf{y}_n, \mathbf{y}_{n+1})<+\infty,$ then
  $\sum_{n \in \N} d_i(\mathbf{y}_n(i), \mathbf{y}_{n+1}(i))<+\infty$ for every $i \in J$; thus
  $d_i(\mathbf{y}_n(i),z(i))\to 0$ as $n \to +\infty$ for some $z(i) \in Y_i$. Since the topology generated by $d$ is the pointwise convergence topology (or by direct computation by using the Bounded Convergence Theorem) we have $d(\mathbf{y}_n, \mathbf{z}) \to 0$ as $n \to +\infty$ where
  $\mathbf{z}(i):=z(i)$ for all $i \in J$; whence $(\prod_{n \in J} Y_n,d)$ is complete. 
  Conversely suppose that $(\prod_{n \in J} Y_n,d)$ is complete and fix $j \in J$. Fix also $\mathbf{z} \in  \prod_{n \in J} Y_n$ and suppose that $\sum_{n \in \mathbb{N}} d_j(y_i,y_{i+1}) <+\infty$ where
  $\{y_i\}_{i \in \N}$ is a sequence in $Y_j$. For every fixed $i \in \N$, define $\mathbf{y}_i$ as $\mathbf{y}_i(n):=\mathbf{z}(n)$ for all $n \neq j$  and $\mathbf{y}_i(j):=y_i$. Then $\sum_{i\in \N}
  d(\mathbf{y}_i,\mathbf{y}_{i+1})=\alpha_j \sum_{i \in \N} d_j(y_i, y_{i+1})<+\infty$ whence 
  $d(\mathbf{y}_i,\mathbf{w})\to 0$ as $i \to +\infty$ for some $\mathbf{w} \in \prod_{n \in J} Y_n$ which implies $d_j(y_i, \mathbf{w}(j))=d_j(\mathbf{y}_i(j),\mathbf{w}(j)) \to 0$ as $i \to +\infty$. This proves that $(Y_j, d_j)$ is complete.
 
 Thus, every finite metric space $(Y_i,d_i)$ is Polish if and only if $\prod_{n \in J} Y_n$ is a Polish metric space with the distance defined by equation~\eqref{eq:productdistance}. This applies for instance
 to
$\mathbb{R}^X$ endowed with the distance
 \[
 d(\mathbf{z},\mathbf{v} ):=
 \sum_{n \in J} \frac{\min(|\mathbf{z}(n)-\mathbf{v}(n)|,1)}{2^n}  
 \]
 where $\{x_i \colon i \in J\}$ is a (finite or infinite) enumeration of $X$ and $J:=\{1, \ldots, \# X\}$.
 Whence $\mathbb{R}^X$
 is a Polish metric space and the metric $d$ generates the pointwise convergence topology.
%
%

Since $[0,1]^X$ and $\mathbb{N}^X$ are closed subsets of $\mathbb{R}^X$, they are Polish metric spaces as well. In particular every measure $\mu_x$, supported on $S_X \subseteq \mathbb{N}^X$, can be seen as a measure defined on 
$\mathbb{N}^X$ or $\mathbb{R}^X$.
 	We note that $\mathbb{R}^X$ is a partially ordered Polish metric space, meaning that the set $\{(\mathbf{z}, \mathbf{v}) \in\mathbb{R}^X \times \mathbb{R}^X \colon \mathbf{z} \le \mathbf{v}\}$ is a closed subset of $\mathbb{R}^X \times \mathbb{R}^X$. 

 \end{rem}
 %

 		\section*{Acknowledgements}
 The authors are grateful to Elisabetta Candellero for useful discussions. They acknowledge support by INdAM--GNAMPA.


 \end{document}